\documentclass[a4paper]{amsart}

\usepackage[mathscr]{eucal}
\usepackage{amsmath,amssymb,amsbsy,amsthm,amsfonts,comment,bbm}
\usepackage{hyperref}
\hypersetup{
  pdfauthor = {Morimichi Kawasaki and Ryuma Orita},
  pdfkeywords = {Symplectic manifolds, groups of Hamiltonian diffeomorphisms, fragmentation norms, spectral invariants, quasi-morphisms},
  pdftitle = {Disjoint superheavy subsets and fragmentation norms},
  pdfpagemode = UseNone,
  bookmarksnumbered=true,
}

\title{Disjoint superheavy subsets and fragmentation norms}

\author{Morimichi Kawasaki} 
\address[Morimichi Kawasaki]{Research Institute for Mathematical Sciences, Kyoto University, Kyoto 606-8502, Japan}
\email{kawasaki@kurims.kyoto-u.ac.jp}

\author{Ryuma Orita} 
\address[Ryuma Orita]{Department of Mathematical Sciences, Tokyo Metropolitan University, Tokyo 192-0397, Japan}
\email{ryuma.orita@gmail.com}
\urladdr{https://ryuma-orita.github.io/}

\date{\today}

\thanks{The first named author has been supported by IBS-R003-D1. This work has been supported by JSPS KAKENHI Grant Numbers JP18J00765, JP18J00335.}

\newtheorem{theorem}{Theorem}[section]
\newtheorem{lemma}[theorem]{Lemma}
\newtheorem{proposition}[theorem]{Proposition}
\newtheorem{corollary}[theorem]{Corollary}
\newtheorem{problem}[theorem]{Problem}

\theoremstyle{definition}
\newtheorem{definition}[theorem]{Definition}

\theoremstyle{remark}
\newtheorem{remark}[theorem]{Remark}

\newcommand{\osc}{\mathop{\mathrm{osc}}\nolimits}
\newcommand{\Ham}{\mathop{\mathrm{Ham}}\nolimits}
\newcommand{\Vol}{\mathop{\mathrm{Vol}}\nolimits}
\newcommand{\QH}{\mathop{\mathrm{QH}}\nolimits}
\newcommand{\HF}{\mathop{\mathrm{HF}}\nolimits}
\newcommand{\supp}{\mathop{\mathrm{supp}}\nolimits}

\newcommand{\relmiddle}[1]{\mathrel{}\middle#1\mathrel{}}

\begin{document}

\begin{abstract}
We present a lower bound for a fragmentation norm and construct a bi-Lipschitz embedding $I\colon \mathbb{R}^n\to\Ham(M)$
with respect to the fragmentation norm on the group $\Ham(M)$ of Hamiltonian diffeomorphisms of a symplectic manifold $(M,\omega)$.
As an application, we provide an answer to Brandenbursky's question on fragmentation norms on $\Ham(\Sigma_g)$, where $\Sigma_g$ is a closed Riemannian surface of genus $g\geq 2$.
\end{abstract}

\maketitle

\tableofcontents


\section{Introduction}\label{section:intro}

\subsection{Background and definition}

Let $(M,\omega)$ be a symplectic manifold.
Let $\Ham(M)$ denote the group of compactly supported Hamiltonian diffeomorphisms of $M$.
In his well-known work \cite{Ba}, Banyaga proved the simplicity of $\Ham(M)$ when $M$ is a closed symplectic manifold.
The key ingredient was the proof of the fragmentation lemma for this group, which, in turn,
allows us to define fragmentation norms on $\Ham(M)$ as follows.
Let $\mathcal{U}=\{U_{\lambda}\}_{\lambda}$ be an open covering of $M$.
The fragmentation lemma implies that for every $\phi\in\Ham(M)$ there exists a positive integer $n$ such that $\phi$ can be represented as a product of $n$ diffeomorphisms
$\theta_i\in\Ham(U_{\lambda_i})$, where $\lambda_i\in\lambda$ and $1\leq i\leq n$.
For $\phi\neq\mathrm{id}_M$,
its \textit{fragmentation norm} $\|\phi\|_{\mathcal{U}}$ with respect to the covering $\mathcal{U}$ is defined to be the minimal number of factors in such a product.
We set $\|\phi\|_{\mathcal{U}}=0$ when $\phi=\mathrm{id}_M$.
Accordingly, the fragmentation norm with respect to an open subset $U$ of $M$ is defined as follows.
We consider an open covering $\mathcal{U}_U$ consisting of all open subsets $V$ such that $\psi(V)\subset U$ for some $\psi\in\Ham(M)$.
The fragmentation norm $\|\phi\|_U$ of $\phi$ is defined to be $\|\phi\|_{\mathcal{U}_U}$.

Entov and Polterovich \cite{EP03} provided a lower bound for the quantitative fragmentation norm \cite{EP03},
using primarily the Oh--Schwarz spectral invariant constructed using Hamiltonian Floer homology.
Subsequently, Burago, Ivanov, and Polterovich \cite{BIP} provided a lower bound for the fragmentation norm itself,
also using the Oh--Schwarz spectral invariant, but their argument had a different basis; see also \cite[Section 4.4]{E}.
In addition, Lanzat \cite{L} and Monzner, Vichery, and Zapolsky \cite{MVZ} provided lower bounds for the fragmentation norms in the case in which $M$ is an open symplectic manifold,
basing their strategies on arguments from \cite{EP03}.
In addition, Brandenbursky and Brandenbursky--K\c{e}dra \cite{Br,BK} provided a lower bound for the fragmentation norm using a Polterovich quasi-morphism whose construction does not involve Floer theory.

Recently, fragmentation norms have been receiving considerable attention,
because they are known to be related to the open problem of the simplicity of the group of compactly supported measure-preserving homeomorphisms of an open disk in the Euclidean plane \cite{LR,EPP}.

In the present paper, we provide a lower bound for the fragmentation norm and construct a bi-Lipschitz embedding $I\colon \mathbb{R}^n\to\Ham(M)$ with respect to the fragmentation norm on $\Ham(M)$.
Our strategy of the proof is based on the work of Entov, Polterovich, and Py \cite{EPP}.
As an application, we provide an answer to Brandenbursky's question \cite[Remark 1.5]{Br}.
The solution involves both Hamiltonian and Lagrangian Floer theory.


\subsection{Principal results}

Let $(M,\omega)$ be a symplectic manifold.
Let $\widetilde{\Ham}(M)$ denote the universal cover of $\Ham(M)$.
Here we define subadditive invariants on $\widetilde{\Ham}(M)$
as a generalization of the Oh--Schwarz spectral invariant and the Lagrangian spectral invariant.

\begin{definition}
A function $c\colon\widetilde{\Ham}(M)\to\mathbb{R}$ is called a \textit{subadditive invariant} if it satisfies the subadditivity condition, i.e.,
$c(\tilde{\phi}\tilde{\psi})\leq c(\tilde{\phi})+c(\tilde{\psi})$ for any $\tilde{\phi},\tilde{\psi}\in\widetilde{\Ham}(M)$.
\end{definition}

\begin{remark}
Polterovich and Rosen introduced a function similar to our subadditive invariant \cite[Section 3.4]{PR}.
However, in addition to subadditivity, they assumed conjugation invariance.
In this paper, we do not make that assumption, because in Section \ref{section:lagspecinv}, we deal with Lagrangian spectral invariants,
which are not conjugation invariant.
\end{remark}

Let $N$ be a positive integer.
The \textit{oscillation norm} $\osc$ on $\mathbb{R}^N$ is defined to be $\osc(r_1,\ldots,r_N)=\max_{i,j}\lvert r_i-r_j\rvert$
for $(r_1,\ldots,r_N)\in\mathbb{R}^N$.
We refer to Section \ref{section:preliminaries} for the definitions of the notions appearing in the following principal theorems.

\begin{theorem}\label{theorem:main}
Let $c_0,c_1,\ldots,c_N\colon\widetilde{\Ham}(M)\to\mathbb{R}$ be subadditive invariants descending asymptotically to $\Ham(M)$.
Let $U$ be an open subset of $M$ satisfying the normally bounded spectrum condition with respect to $c_i$ for all $i=0,1,\ldots,N$.
Let $X_0,X_1,\ldots,X_N$ be mutually disjoint closed subsets of $M$ such that each $X_i$ is $c_i$-superheavy.
Then there exists a bi-Lipschitz injective homomorphism
\[
	I\colon(\mathbb{R}^N,\osc)\to(\Ham(M),\|\cdot\|_U).
\]
\end{theorem}

When $c_0,c_1,\ldots,c_N$ are quasi-morphisms, we obtain a stronger result than Theorem \ref{theorem:main}.

\begin{theorem}\label{theorem:main2}
Let $c_0,c_1,\ldots,c_N\colon\widetilde{\Ham}(M)\to\mathbb{R}$ be subadditive invariants descending asymptotically to $\Ham(M)$ that are quasi-morphisms.
Let $U$ be an open subset of $M$ satisfying the asymptotically vanishing spectrum condition with respect to $c_i$ for all $i=0,1,\ldots,N$.
Let $X_0,X_1,\ldots,X_N$ be mutually disjoint closed subsets of $M$ such that each $X_i$ is $c_i$-superheavy.
Then, there exists a bi-Lipschitz injective homomorphism
\[
	I\colon(\mathbb{R}^N,\osc)\to(\Ham(M),\|\cdot\|_U).
\]
\end{theorem}

Concerning the fragmentation norm $\|\cdot\|_{\mathcal{U}}$ with respect to an open covering $\mathcal{U}$,
we have the following theorem.

\begin{theorem}\label{theorem:main3}
Let $c_0,c_1,\ldots,c_N\colon\widetilde{\Ham}(M)\to\mathbb{R}$ be subadditive invariants descending asymptotically to $\Ham(M)$.
Let $\mathcal{U}=\{U_{\lambda}\}_{\lambda}$ be an open covering of $M$
such that each $U_{\lambda}$ satisfies the bounded spectrum condition with respect to $c_i$ for all $i=0,1,\ldots,N$.
Let $X_0,X_1,\ldots,X_N$ be mutually disjoint closed subsets of $M$ such that each $X_i$ is $c_i$-superheavy.
Then, there exists a bi-Lipschitz injective homomorphism
\[
	I\colon(\mathbb{R}^N,\osc)\to(\Ham(M),\|\cdot\|_{\mathcal{U}}).
\]
\end{theorem}

We prove Theorems \ref{theorem:main}, \ref{theorem:main2}, and \ref{theorem:main3} in Section \ref{section:pf of thm}.


\section{Applications}\label{section:app}

In this section, we provide applications of our principal theorems.

Let $(M,\omega)$ be a symplectic manifold and $X$ a subset of $M$.
An open subset $U\subset M$ is called \textit{displaceable from $X$}
if there exists a Hamiltonian $H\colon S^1\times M\to\mathbb{R}$ such that $\varphi_H(U)\cap\overline{X}=\emptyset$,
where $\varphi_H$ is the Hamiltonian diffeomorphism generated by $H$ and $\overline{X}$ is the topological closure of $X$.
$U\subset M$ is called \textit{$($abstractly$)$ displaceable} if $U$ is displaceable from $U$ itself.

\subsection{$B^{2n}$}

We consider the $2n$-dimensional ball
\[
	B^{2n}=\left\{\,(p,q)\in\mathbb{R}^n\times\mathbb{R}^n\relmiddle|\lvert p\rvert^2+\lvert q\rvert^2<1\,\right\}\subset\mathbb{R}^{2n}
\]
equipped with the symplectic form $dp_1\wedge dq_1+\cdots+dp_n\wedge dq_n$, where $p=(p_1,\ldots,p_n)$ and $q=(q_1,\ldots,q_n)$.
We have the following corollary of Theorem \ref{theorem:main2}.

\begin{corollary}\label{b2n}
For any open ball $B(r)\subset\mathbb{R}^{2n}$ of radius $r<1$ centered at $0\in\mathbb{R}^{2n}$
and any positive integer $N$,
there exists a bi-Lipschitz injective homomorphism
\[
	I\colon(\mathbb{R}^N,\osc)\to(\Ham(B^{2n}),\|\cdot\|_{B(r)}).
\]
\end{corollary}

We prove Corollary \ref{b2n} in Section \ref{section:pf of cor}.

\begin{remark}
Entov, Polterovich, and Py implicitly proved a similar statement when $r$ is sufficiently small \cite{EPP}.
\end{remark}


\subsection{$S^2\times S^2$}

We consider the product $S^2\times S^2$ with the symplectic form $\mathrm{pr}_1^{\ast}\omega_1+\mathrm{pr}_2^{\ast}\omega_1$,
where $\omega_1$ is a symplectic form on $S^2$ with $\int_{S^2}\omega_1=1$ and $\mathrm{pr}_1, \mathrm{pr}_2\colon S^2\times S^2\to S^2$ are the first and second projections, respectively.
Let $E$ denote the equator of $S^2$.

We have the following corollary of Theorem \ref{theorem:main2}

\begin{corollary}\label{s2s2}
Let $U$ be an open subset of $S^2\times S^2$ that is either abstractly displaceable or displaceable from $E\times E$.
Then, for any positive integer $N$, there exists a bi-Lipschitz injective homomorphism
\[
	I\colon(\mathbb{R}^N,\osc)\to(\Ham(S^2\times S^2),\|\cdot\|_U).
\]
\end{corollary}

We prove Corollary \ref{s2s2} in Section \ref{section:pf of cor}.


\subsection{$\mathbb{C}P^2$}

Let $(\mathbb{C}P^2,\omega_{\mathrm{FS}})$ be two-dimensional complex projective space equipped with the Fubini--Study form.
Then, the real projective space $\mathbb{R}P^2$ is naturally embedded in $(\mathbb{C}P^2,\omega_{\mathrm{FS}})$ as a Lagrangian submanifold.
The \textit{Clifford torus} $L_C$ is the Lagrangian submanifold
\[
	L_C=\left\{\,[z_0:z_1:z_2]\in\mathbb{C}P^2\relmiddle|\lvert z_0\rvert=\lvert z_1\rvert=\lvert z_2\rvert\,\right\}.
\]
By \cite{BEP}, $L_C$ is a \textit{stem} in the sense of \cite[Definition 2.3]{EP06}.
There is another Lagrangian submanifold $L_W$ constructed by Wu \cite{W} that is disjoint from $\mathbb{R}P^2$.
We call $L_W$ the \textit{Chekanov torus}.
Although there are some other Lagraingian submanifolds of $\mathbb{C}P^2$ called the Chekanov torus \cite{CS,Gad,BC},
Oakley and Usher proved that they are all Hamiltonian isotopic \cite{OU}.

We have the following corollary of Theorem \ref{theorem:main2}.

\begin{corollary}\label{cp2}
Let $U$ be an open subset of $\mathbb{C}P^2$ satisfying one of the following conditions:
\begin{enumerate}
	\item $U$ is abstractly displaceable.
	\item $U$ is displaceable from $\mathbb{R}P^2$ and $L_W$.
	\item $U$ is displaceable from $L_C$.
\end{enumerate}
Then, there exists a bi-Lipschitz injective homomorphism
\[
	I\colon(\mathbb{R},\lvert\cdot\rvert)\to(\Ham(\mathbb{C}P^2),\|\cdot\|_U).
\]
\end{corollary}

Here $\lvert\cdot\rvert$ is the absolute value.


\subsection{Surfaces}

Let $(\Sigma_g,\omega)$ be a closed Riemannian surface $\Sigma_g$ of genus $g$ with an area form $\omega$.
Brandenbursky studied fragmentation norms on $\Ham(\Sigma_g)$ under some assumptions.

\begin{theorem}[{\cite[Theorem 4]{Br}}]\label{Brandenbursky}
Let $g$ be a positive integer with $g\geq 2$ and $U$ be a contractible open subset of $\Sigma_g$.
Then, there exists a bi-Lipschitz injective homomorphism
\[
	I\colon(\mathbb{Z}^{2g-2},\|\cdot\|_{\mathrm{word}})\to(\Ham(\Sigma_g),\|\cdot\|_U).
\]
\end{theorem}

Here $\|\cdot\|_{\mathrm{word}}$ is the word metric on $\mathbb{Z}^{2g-2}$.
We point out that Burago, Ivanov, and Polterovich \cite{BIP} already proved the existence of a bi-Lipschitz injective homomorphism
$I\colon(\mathbb{Z},\lvert\cdot\rvert)\to(\Ham(\Sigma_g),\|\cdot\|_U)$, where $g$ is positive and $U$ is displaceable.
Relating to Theorem \ref{Brandenbursky}, Brandenbursky asked whether one can construct a bi-Lipschitz injective homomorphism $\mathbb{Z}^N\to\Ham(\Sigma_g)$ for any $N$ and any $g\geq 2$ \cite[Remark 1.5]{Br}.
As a corollary of Theorem \ref{theorem:main}, we solve his problem and generalize Theorem \ref{Brandenbursky}.

\begin{corollary}\label{surface main theorem}
Let $g$ be a positive integer.
Let $U$ be a contractible open subset of $\Sigma_g$ and $N$ a positive integer.
Then, there exists a bi-Lipschitz injective homomorphism
\[
	I\colon(\mathbb{R}^N,\osc)\to(\Ham(\Sigma_g),\|\cdot\|_U).
\]
\end{corollary}

\begin{remark}
Since all norms on a finite-dimensional vector space are equivalent,
the restriction of $\osc$ to $\mathbb{Z}^{2g-2}\subset\mathbb{R}^{2g-2}$ is equivalent to the word metric on $\mathbb{Z}^{2g-2}$.
\end{remark}

Moreover, as a corollary of Theorem \ref{theorem:main3}, we prove the following result.

\begin{corollary}\label{annulus frag}
Let $g$ be a positive integer and $C$ be a non-contractible simple closed curve in $\Sigma_g$.
Let $\mathcal{U}=\{U_\lambda\}_{\lambda}$ be an open covering such that each $U_\lambda$ is displaceable from $C$.
Then, for any positive integer $N$, there exists a bi-Lipschitz injective homomorphism
\[
	I\colon(\mathbb{R}^N,\osc)\to(\Ham(\Sigma_g),\|\cdot\|_{\mathcal{U}}).
\]
\end{corollary}

We prove Corollaries \ref{surface main theorem} and \ref{annulus frag} in Section \ref{section:pf of cor}.

Let $\mathbb{N}$ denote the set of positive integers.
For $\vec{g}=(g_1,\ldots,g_n)\in\mathbb{N}^n$,
let $(\Sigma_{\vec{g}},\omega)$ denote the product manifold $\Sigma_{\vec{g}}=\Sigma_{g_1}\times\cdots\times\Sigma_{g_n}$ with a symplectic form $\omega$.
Entov and Polterovich constructed a partial Calabi quasi-morphism (see Section \ref{section:calabi} for the definition) on $\Ham(\Sigma_{\vec{g}})$ for any $\vec{g}\in\mathbb{N}^n$  by using the Oh--Schwarz spectral invariant \cite{EP06}.
They asked whether one can construct a Calabi quasi-morphism on $\Ham(\Sigma_g)$ for positive $g$.
Py gave a positive answer to their question.
Moreover, he constructed an infinite family of linearly independent Calabi quasi-morphisms on $\Ham(\Sigma_g)$ for positive $g$ \cite{Py6, Py6-2}.
Brandenbursky provided another construction of such an infinite family for $g\geq 2$ \cite{Br}.
Brandenbursky, Kedra, and Shelukhin \cite{BKS} also provided a construction of Calabi quasi-morphisms in case $g=1$.
In this paper, we prove the following theorem.

\begin{theorem}\label{calabi on surface}
For any $\vec{g}\in\mathbb{N}^n$,
the dimension of the space of partial Calabi quasi-morphisms on $\Ham(\Sigma_{\vec{g}})$ is infinite.
\end{theorem}

We prove Theorem \ref{calabi on surface} in Section \ref{section:calabi}.


\section{Preliminaries}\label{section:preliminaries}

In this section, we provide the defnitions appearing in Sections \ref{section:intro} and \ref{section:app},
and review their properties.
Let $(M,\omega)$ be a $2n$-dimensional symplectic manifold.

\subsection{Conventions and notation}

For a Hamiltonian $H\colon S^1\times M\to\mathbb{R}$ with compact support, we set $H_t=H(t,\cdot)$ for $t\in S^1$.
The \textit{mean value} of $H$ is defined to be
\[
	\langle H\rangle=\Vol(M)^{-1}\int_0^1\int_M H_t\omega^n\,dt,
\]
where $\Vol(M)=\int_M\omega^n$ is the volume of $(M,\omega)$.
A Hamiltonian $H$ is called \textit{normalized} if $\langle H\rangle=0$.
The \textit{Hamiltonian vector field} $X_{H_t}$ associated with $H_t$ is a time-dependent vector field defined by the formula
\[
	\omega(X_{H_t},\cdot)=-dH_t.
\]
The \textit{Hamiltonian isotopy} $\{\varphi_H^t\}_{t\in\mathbb{R}}$ associated with $H$ is defined by
\[
	\begin{cases}
		\varphi_H^0=\mathrm{id}_M,\\
		\frac{d}{dt}\varphi_H^t=X_{H_t}\circ\varphi_H^t\quad \text{for all}\ t\in\mathbb{R},
	\end{cases}
\]
and its time-one map $\varphi_H=\varphi_H^1$ is referred to as the \textit{Hamiltonian diffeomorphism $($with compact support$)$} generated by $H$.
Let $\Ham(M)$ and $\widetilde{\Ham}(M)$ denote the group of Hamiltonian diffeomorphisms of $M$ with compact support and its universal cover, respectively.
An element of $\widetilde{\Ham}(M)$ is represented by a path in $\Ham(M)$ starting from the identity.
Hence, for every Hamiltonian $H\colon S^1\times M\to\mathbb{R}$ with compact support, its Hamiltonian isotopy $\{\varphi_H^t\}_{t\in\mathbb{R}}$ defines an element $\tilde{\varphi}_H\in\widetilde{\Ham}(M)$.
Let $\mathbbm{1}$ denote the identity of $\widetilde{\Ham}(M)$, i.e., the homotopy class of the constant path $t\mapsto\mathrm{id}_M$ in $\Ham(M)$.

For an open subset $U$ of $M$,
let $\mathcal{H}(U)$ be the subset of $C^{\infty}(S^1\times M)$ consisting of all Hamiltonians supported in $S^1\times U$.


\subsection{Subadditive invariants and superheaviness}

Let $c\colon\widetilde{\Ham}(M)\to\mathbb{R}$ be a subadditive invariant.
We define a map $\sigma_c\colon\widetilde{\Ham}(M)\to\mathbb{R}$ as
\[
	\sigma_c(\tilde{\phi})=\lim_{k\to\infty}\frac{c(\tilde{\phi}^k)}k.
\]
The limit exists by subadditivity property.

\begin{definition}\label{definition:descend}
Let $\pi\colon\widetilde{\Ham}(M)\to\Ham(M)$ denote the natural projection.
\begin{enumerate}
\item We say that a subadditive invariant $c\colon\widetilde{\Ham}(M)\to\mathbb{R}$ \textit{descends to $\Ham(M)$}
	if $c$ induces a map $\bar{c}\colon\Ham(M)\to\mathbb{R}$ such that $c=\bar{c}\circ\pi$.
\item We say that a subadditive invariant $c\colon\widetilde{\Ham}(M)\to\mathbb{R}$ \textit{descends asymptotically to $\Ham(M)$}
	if the map $\sigma_c$ induces a map $\bar{\sigma}_c\colon\Ham(M)\to\mathbb{R}$ such that $\sigma_c=\bar{\sigma}_c\circ\pi$.
\end{enumerate}
\end{definition}

By definition, every subadditive invariant descending to $\Ham(M)$ descends asymptotically to $\Ham(M)$.

Given two subadditive invariants, we can prove the following proposition.

\begin{proposition}\label{comparing descending}
Let $c,c'\colon\widetilde{\Ham}(M)\to\mathbb{R}$ be subadditive invariants.
Assume that $c'$ descends asymptotically to $\Ham(M)$ and $c(\tilde{\varphi}_H)\leq c'(\tilde{\varphi}_H)$ holds for any Hamiltonian $H\colon S^1\times M\to\mathbb{R}$.
Then, $c$ also descends asymptotically to $\Ham(M)$.
\end{proposition}

To prove Proposition \ref{comparing descending}, we first prove the following lemma.

\begin{lemma}\label{equiv of descending}
Let $c\colon\widetilde{\Ham}(M)\to\mathbb{R}$ be a subadditive invariant.
Then, $c$ descends asymptotically to $\Ham(M)$ if and only if $\sigma_c|_{\pi_1(\Ham(M))}=0$.
\end{lemma}

\begin{proof}
The ``only if'' part follows immediately from the definition of descending asymptotically.
Accordingly, we prove the ``if'' part and assume that $\sigma_c|_{\pi_1(\Ham(M))}=0$.
Take $\tilde{\phi}\in\widetilde{\Ham}(M)$ and $\tilde{\psi}\in\pi_1\bigl(\Ham(M)\bigr)$.
By subadditivity, for any positive integer $k$,
\begin{equation}\label{eq:beforedivide}
	c(\tilde{\phi}^k)-c(\tilde{\psi}^{-k})\leq c(\tilde{\phi}^k\tilde{\psi}^k)\leq c(\tilde{\phi}^k)+c(\tilde{\psi}^k).
\end{equation}
Since $\pi_1\bigl(\Ham(M)\bigr)$ is a connected topological group with respect to the $C^{\infty}$-topology,
$\pi_1\bigl(\Ham(M)\bigr)$ lies in the center of $\widetilde{\Ham}(M)$.
Here note that the fundamental group $\pi_1(G)$ of a connected topological group $G$ lies in the center of its universal cover $\widetilde{G}$ (see, for example, \cite[Theorem 15]{P}).
Hence, $c(\tilde{\phi}^k\tilde{\psi}^k)=c\bigl((\tilde{\phi}\tilde{\psi})^k\bigr)$.
Dividing \eqref{eq:beforedivide} by $k$ and passing to the limit as $k\to\infty$ yields
\[
	\sigma_c(\tilde{\phi})-\sigma_c(\tilde{\psi}^{-1})\leq \sigma_c(\tilde{\phi}\tilde{\psi})\leq \sigma_c(\tilde{\phi})+\sigma_c(\tilde{\psi}).
\]
Since $\tilde{\psi}\in\pi_1\bigl(\Ham(M)\bigr)$ and $\sigma_c|_{\pi_1(\Ham(M))}=0$, we conclude that $\sigma_c(\tilde{\phi}\tilde{\psi})=\sigma_c(\tilde{\phi})$.
Since $\sigma_c(\tilde{\phi}\tilde{\psi})=\sigma_c(\tilde{\phi})$ for any $\tilde{\phi}\in\widetilde{\Ham}(M)$ and any $\tilde{\psi}\in\pi_1\bigl(\Ham(M)\bigr)$,
$c$ descends asymptotically to $\Ham(M)$.
\end{proof}

\begin{proof}[Proof of Proposition \ref{comparing descending}]
By Lemma \ref{equiv of descending}, it is sufficient to show that $\sigma_c|_{\pi_1(\Ham(M))}=0$.
Let $H\colon S^1\times M\to\mathbb{R}$ be a Hamiltonian generating $\mathrm{id}_M\in\Ham(M)$.
Then, $\tilde{\varphi}_H\tilde{\varphi}_H^{-1}=\mathbbm{1}$.
By subadditivity,
\[
	c(\mathbbm{1})\leq c(\tilde{\varphi}_H^k)+c(\tilde{\varphi}_H^{-k}).
\]
Dividing by $k$ and passing to the limit as $k\to\infty$ yields
\[
	0\leq \sigma_c(\tilde{\varphi}_H)+\sigma_c(\tilde{\varphi}_H^{-1}).
\]
Since $c'$ descends asymptotically to $\Ham(M)$,
\[
	\sigma_c(\tilde{\varphi}_H)\leq\sigma_{c'}(\tilde{\varphi}_H)=\bar{\sigma}_{c'}(\mathrm{id}_M)=0.
\]
Similarly,
\[
	\sigma_c(\tilde{\varphi}_H^{-1})\leq\sigma_{c'}(\tilde{\varphi}_H^{-1})=\bar{\sigma}_{c'}(\mathrm{id}_M)=0.
\]
Thus,
\[
	\sigma_c(\tilde{\varphi}_H)=\sigma_c(\tilde{\varphi}_H^{-1})=0.\qedhere
\]
\end{proof}

\begin{definition}
A closed subset $X$ of $M$ is called \textit{$c$-superheavy} if
\[
	\inf_{S^1\times X}H\leq\sigma_c(\tilde{\varphi}_H)\leq\sup_{S^1\times X}H
\]
for any normalized Hamiltonian $H\colon S^1\times M\to\mathbb{R}$.
\end{definition}

By definition, we have the following result.

\begin{proposition}\label{proposition:shv}
Let $X$ be a $c$-superheavy subset of $M$.
Then, for any $\alpha\in\mathbb{R}$ and any normalized Hamiltonian $H\colon S^1\times M\to\mathbb{R}$ with $H|_{S^1\times X}\equiv \alpha$,
\[
	\sigma_c(\tilde{\varphi}_H)=\alpha.
\]
\end{proposition}


\subsection{Spectrum conditions}

We define three kinds of spectrum conditions.
We recall that the mean value $\langle H\rangle$ of a Hamiltonian $H\colon S^1\times M\to\mathbb{R}$ is given by
\[
	\langle H\rangle=\Vol(M)^{-1}\int_0^1\int_M H_t\omega^n\,dt.
\]

\subsubsection{Normally bounded spectrum condition}

Let $c\colon\widetilde{\Ham}(M)\to\mathbb{R}$ be a subadditive invariant.

\begin{definition}\label{definition:NBSC}
An open subset $U$ of $M$ satisfies the \textit{normally bounded spectrum condition with respect to $c$} if
there exists a positive number $K>0$ such that
for any $F\in\mathcal{H}(U)$ and any $\tilde{\psi}\in\widetilde{\Ham}(M)$,
\begin{equation}\label{eq:NBSC}
	c(\tilde{\psi}^{-1}\tilde{\varphi}_F\tilde{\psi})+\langle F\rangle\leq K.
\end{equation}
\end{definition}

\begin{remark}
If we put $c(H)=c(\tilde{\varphi}_H)+\langle H\rangle$ for a Hamiltonian $H\colon S^1\times M\to\mathbb{R}$,
then the inequality \eqref{eq:NBSC} can be written as $c(F\circ\psi)\leq K$.
However, in this paper, we avoid this notation for simplicity.
\end{remark}

\begin{remark}
Definition \ref{definition:NBSC} is equivalent to the existence of a positive number $K>0$ such that
for any $\psi\in\Ham(M)$ and any $F\in\mathcal{H}\bigl(\psi(U)\bigr)$,
\begin{equation}\label{eq:NBSC2}
	c(\tilde{\varphi}_F)+\langle F\rangle\leq K
\end{equation}
since the Hamiltonian diffeomorphism generated by $F\circ\psi$ is $\psi^{-1}\varphi_F\psi$.
\end{remark}

\begin{remark}\label{MVZ}
When $c$ is an Oh--Schwarz spectral invariant,
the normally bounded spectrum condition is equivalent to the bounded spectrum condition (see Definition \ref{definition:BSC})
since Oh--Schwarz spectral invariants are conjugation invariant.
The normally bounded spectrum condition was introduced by Monzner, Vichery, and Zapolsky \cite{MVZ}. 
\end{remark}

\begin{proposition}\label{proposition:VPNBSC}
Let $U$ be an open subset of $M$ satisfying the normally bounded spectrum condition with respect to $c$.
For any $\psi\in\Ham(M)$ and any $F\in\mathcal{H}\bigl(\psi(U)\bigr)$,
\[
	\sigma_c(\tilde{\varphi}_F)=-\langle F\rangle.
\]
\end{proposition}

\begin{proof}
Let $\psi\in\Ham(M)$ and $F\in\mathcal{H}\bigl(\psi(U)\bigr)$.
Note that the Hamiltonian $-F\circ\varphi_F$ generates $\tilde{\varphi}_F^{-1}$
and satisfies $\langle -F\circ\varphi_F\rangle=-\langle F\rangle$.
Since $U$ satisfies the normally bounded spectrum condition with respect to $c$,
we can choose a positive number $K>0$ such that for any $k\in\mathbb{Z}$,
\[
	c(\tilde{\varphi}_F^k)+\langle kF\rangle\leq K \quad\text{and}\quad c(\tilde{\varphi}_F^{-k})+\langle -kF\rangle\leq K.
\]
By subadditivity,
\[
	c(\mathbbm{1})\leq c(\tilde{\varphi}_F^k)+c(\tilde{\varphi}_F^{-k}).
\]
Therefore,
\[
	-K+c(\mathbbm{1})\leq -c(\tilde{\varphi}_F^{-k})-\langle -kF\rangle+c(\mathbbm{1})\leq c(\tilde{\varphi}_F^k)+\langle kF\rangle\leq K.
\]
Dividing by $k$ and passing to the limit as $k\to\infty$ yields
\[
	\sigma_c(\tilde{\varphi}_F)+\langle F\rangle=\lim_{k\to\infty}\frac{c(\tilde{\varphi}_F^k)+k\langle F\rangle}k=0.\qedhere
\]
\end{proof}

The following proposition is useful in the proof of Theorem \ref{theorem:main}.

\begin{proposition}\label{proposition:defect1}
Let $U$ be an open subset of $M$ satisfying the normally bounded spectrum condition with respect to $c$.
Then, there exists a positive number $K>0$ such that
for any $\tilde{\phi}\in\widetilde{\Ham}(M)$, any $\psi\in\Ham(M)$ and any $F\in\mathcal{H}\bigl(\psi(U)\bigr)$,
\[
	\left\lvert\sigma_c(\tilde{\varphi}_F\tilde{\phi})-\sigma_c(\tilde{\varphi}_F)-\sigma_c(\tilde{\phi})\right\rvert\leq K.
\]
\end{proposition}

To prove Proposition \ref{proposition:defect1}, we first prove the following lemma.

\begin{lemma}\label{lemma:defect}
There exists a positive number $K>0$ such that
for any $\tilde{\phi}\in\widetilde{\Ham}(M)$, any $\psi\in\Ham(M)$ and any $F\in\mathcal{H}\bigl(\psi(U)\bigr)$,
\[
	\left\lvert c(\tilde{\varphi}_F\tilde{\phi})-c(\tilde{\phi})+\langle F\rangle\right\rvert\leq K.
\]
\end{lemma}

\begin{proof}
Let $\tilde{\phi}\in\widetilde{\Ham}(M)$, $\psi\in\Ham(M)$ and $F\in\mathcal{H}\bigl(\psi(U)\bigr)$.
Since $U$ satisfies the normally bounded spectrum condition with respect to $c$,
we can choose a positive number $K>0$ such that
\[
	c(\tilde{\varphi}_F)+\langle F\rangle\leq K\quad\text{and}\quad c(\tilde{\varphi}_F^{-1})-\langle F\rangle\leq K.
\]
By subadditivity,
\[
	c(\tilde{\phi})\leq c(\tilde{\varphi}_F^{-1})+c(\tilde{\varphi}_F\tilde{\phi})\quad\text{and}\quad c(\tilde{\varphi}_F\tilde{\phi})\leq c(\tilde{\varphi}_F)+c(\tilde{\phi}).
\]
Therefore,
\[
	-K\leq -c(\tilde{\varphi}_F^{-1})+\langle F\rangle\leq c(\tilde{\varphi}_F\tilde{\phi})-c(\tilde{\phi})+\langle F\rangle\leq c(\tilde{\varphi}_F)+\langle F\rangle\leq K.\qedhere
\]
\end{proof}

\begin{proof}[Proof of Proposition \ref{proposition:defect1}]
Let $\tilde{\phi}\in\widetilde{\Ham}(M)$, $\psi\in\Ham(M)$ and $F\in\mathcal{H}\bigl(\psi(U)\bigr)$.
For an integer $k$, decompose $(\tilde{\varphi}_F\tilde{\phi})^k$ into
\[
	(\tilde{\varphi}_F\tilde{\phi})^k%
	=\tilde{\varphi}_F(\tilde{\phi}\tilde{\varphi}_F\tilde{\phi}^{-1})\cdots(\tilde{\phi}^{k-1}\tilde{\varphi}_F\tilde{\phi}^{-k+1})\tilde{\phi}^k.
\]
Since $\tilde{\phi}^i\tilde{\varphi}_F\tilde{\phi}^{-i}=\tilde{\varphi}_{F\circ\phi^{-i}}$ for all $i=0,1,\ldots,k-1$, Lemma \ref{lemma:defect} implies that there exists a positive number $K>0$ such that
\[
	\left\lvert c\bigl((\tilde{\varphi}_F\tilde{\phi})^k\bigr)-c\bigl((\tilde{\phi}\tilde{\varphi}_F\tilde{\phi}^{-1})\cdots(\tilde{\phi}^{k-1}\tilde{\varphi}_F\tilde{\phi}^{-k+1})\tilde{\phi}^k\bigr)+\langle F\rangle\right\rvert\leq K,
\]
\[
	\left\lvert c\bigl((\tilde{\phi}\tilde{\varphi}_F\tilde{\phi}^{-1})\cdots(\tilde{\phi}^{k-1}\tilde{\varphi}_F\tilde{\phi}^{-k+1})\tilde{\phi}^k\bigr)-c\bigl((\tilde{\phi}^2\tilde{\varphi}_F\tilde{\phi}^{-2})\cdots(\tilde{\phi}^{k-1}\tilde{\varphi}_F\tilde{\phi}^{-k+1})\tilde{\phi}^k\bigr)+\langle F\circ\phi^{-1}\rangle\right\rvert\leq K,
\]
\[
	\cdots
\]
\[
	\left\lvert c\bigl((\tilde{\phi}^{k-1}\tilde{\varphi}_F\tilde{\phi}^{-k+1})\tilde{\phi}^k\bigr)-c\bigl(\tilde{\phi}^k\bigr)+\langle F\circ\phi^{-k+1}\rangle\right\rvert\leq K.
\]
Therefore, since $\langle F\circ\phi^{-i}\rangle=\langle F\rangle$ for all $i=0,1,\ldots,k-1$, by the triangle inequality,
\begin{align*}
	&\left\lvert c\bigl((\tilde{\varphi}_F\tilde{\phi})^k\bigr)-c(\tilde{\phi}^k)+k\langle F\rangle\right\rvert\\
	&=\left\lvert c\bigl((\tilde{\varphi}_F\tilde{\phi})^k\bigr)-c(\tilde{\phi}^k)+\langle F\rangle+\langle F\circ\phi^{-1}\rangle+\cdots+\langle F\circ\phi^{-k+1}\rangle\right\rvert\\
	&\leq\left\lvert c\bigl((\tilde{\varphi}_F\tilde{\phi})^k\bigr)-c\bigl((\tilde{\phi}\tilde{\varphi}_F\tilde{\phi}^{-1})\cdots(\tilde{\phi}^{k-1}\tilde{\varphi}_F\tilde{\phi}^{-k+1})\tilde{\phi}^k\bigr)+\langle F\rangle\right\rvert\\
	&+\left\lvert c\bigl((\tilde{\phi}\tilde{\varphi}_F\tilde{\phi}^{-1})\cdots(\tilde{\phi}^{k-1}\tilde{\varphi}_F\tilde{\phi}^{-k+1})\tilde{\phi}^k\bigr)-c\bigl((\tilde{\phi}^2\tilde{\varphi}_F\tilde{\phi}^{-2})\cdots(\tilde{\phi}^{k-1}\tilde{\varphi}_F\tilde{\phi}^{-k+1})\tilde{\phi}^k\bigr)+\langle F\circ\phi^{-1}\rangle\right\rvert\\
	&+\cdots\\
	&+\left\lvert c\bigl((\tilde{\phi}^{k-1}\tilde{\varphi}_F\tilde{\phi}^{-k+1})\tilde{\phi}^k\bigr)-c\bigl(\tilde{\phi}^k\bigr)+\langle F\circ\phi^{-k+1}\rangle\right\rvert\\
	&\leq kK.
\end{align*}
Dividing by $k$ and passing to the limit as $k\to\infty$ yields
\[
	\left\lvert\sigma_c(\tilde{\varphi}_F\tilde{\phi})-\sigma_c(\tilde{\phi})+\langle F\rangle\right\rvert%
	=\lim_{k\to\infty}\frac{\left\lvert c\bigl((\tilde{\varphi}_F\tilde{\phi})^k\bigr)-c(\tilde{\phi}^k)+k\langle F\rangle\right\rvert}k\leq K.
\]
Then, Proposition \ref{proposition:VPNBSC} completes the proof of Proposition \ref{proposition:defect1}.
\end{proof}


\subsubsection{Bounded spectrum condition}

Let $c\colon\widetilde{\Ham}(M)\to\mathbb{R}$ be a subadditive invariant.

\begin{definition}\label{definition:BSC}
An open subset $U$ of $M$ satisfies the \textit{bounded spectrum condition with respect to $c$} if
there exists a positive number $K>0$ such that for any $F\in\mathcal{H}(U)$,
\begin{equation}\label{eq:BSC}
	c(\tilde{\varphi}_F)+\langle F\rangle\leq K.
\end{equation}
\end{definition}

Note that the normally bounded spectrum condition implies the bounded spectrum condition.


\subsubsection{Asymptotically vanishing spectrum condition}

Let $c\colon\widetilde{\Ham}(M)\to\mathbb{R}$ be a subadditive invariant.

\begin{definition}
An open subset $U$ of $M$ satisfies the \textit{asymptotically vanishing spectrum condition with respect to $c$} if for any $F\in\mathcal{H}(U)$,
\begin{equation}
	\sigma_c(\tilde{\varphi}_F)+\langle F\rangle=0.
\end{equation}
\end{definition}

\begin{remark}\label{remark:AVSC}
An argument similar to the proof of Proposition \ref{proposition:VPNBSC} shows that
the asymptotically vanishing spectrum condition is weaker than the bounded spectrum condition.
\end{remark}

\begin{proposition}\label{disp and vanish}
Every open subset of $M$ displaceable from a $c$-superheavy subset satisfies the asymptotically vanishing spectrum condition with respect to $c$.
\end{proposition}



To prove Proposition \ref{disp and vanish}, we first prove

\begin{lemma}\label{sigma  is ci}
For any $\tilde{\phi},\tilde{\psi}\in\widetilde{\Ham}(M)$,
$\sigma_c(\tilde{\phi}^{-1}\tilde{\psi}\tilde{\phi})=\sigma_c(\tilde{\psi})$.
\end{lemma}

\begin{proof}

Let $k$ be an integer.
By subadditivity,
\[
	-c(\tilde{\phi}^{-1})-c(\tilde{\phi})\leq c(\tilde{\phi}^{-1}\tilde{\psi}^k\tilde{\phi})-c(\tilde{\psi}^k)\leq c(\tilde{\phi}^{-1})+c(\tilde{\phi}).
\]
Since $(\tilde{\phi}^{-1}\tilde{\psi}\tilde{\phi})^k=\tilde{\phi}^{-1}\tilde{\psi}^k\tilde{\phi}$,
dividing by $k$ and passing to the limit as $k\to\infty$ yields
\[
	\sigma_c(\tilde{\phi}^{-1}\tilde{\psi}\tilde{\phi})=\sigma_c(\tilde{\psi}).\qedhere
\]
\end{proof}

\begin{proof}[Proof of Proposition \ref{disp and vanish}]
Let $X$ be a $c$-superheavy subset of $M$.
Let $U$ be an open subset displaceable from $X$.
By assumption, we can take $\phi\in\Ham(M)$ such that $\phi(U)\cap X=\emptyset$.
Since $X$ is $c$-superheavy, for any $F\in\mathcal{H}\bigl(\phi(U)\bigr)$,
\[
	0=\inf_{S^1\times X}F\leq\sigma_c(\tilde{\varphi}_F)+\langle F\rangle\leq\sup_{S^1\times X}F=0.
\]
Hence, $\phi(U)$ satisfies the asymptotically vanishing spectrum condition with respect to $c$.
Lemma \ref{sigma is ci} implies that $U$ also satisfies the asymptotically vanishing spectrum condition with respect to $c$.
\end{proof}


\section{Delicate Banyaga fragmentation lemma}

To prove the principal theorems, we use the following folklore lemma which is a slightly delicate version of Banyaga's fragmentation lemma (see also \cite[Lemma 2.1]{Ka3}).

\begin{lemma}\label{lemma:uniform}
Let $(M,\omega)$ be a symplectic manifold, $K$ a compact subset of $M$ and $\mathcal{U}$ an open cover of $M$.
Then, there exists a positive number $N_{K,\mathcal{U}}$ such that $\|\varphi_H\|_\mathcal{U}\leq N_{K,\mathcal{U}}$
for any $C^1$-small Hamiltonian $H\colon [0,1]\times M\to\mathbb{R}$ with $\supp(H)\subset [0,1]\times K$.
\end{lemma}

\begin{proof}
Since $K$ is compact, we can take finite open coverings $\mathcal{V}=\{V_i\}_{i=1,\ldots,\ell}$ and $\mathcal{V}'=\{V'_i\}_{i=1,\ldots,\ell}$ of $K$ such that
\begin{itemize}
	\item  for any $i$, $\overline{V_i}\subset V'_i$,
	\item  for any $i$ there exists $U_i\in\mathcal{U}$ such that $\overline{V'_i}\subset U_i$.
\end{itemize}

Take a partition of unity $\{\rho_i\colon K\to[0,1]\}_{i=1,\ldots,\ell}$ subordinated to $\mathcal{V}$
(i.e., $\supp(\rho_i)\subset V_i$ for any $i$).
We then define functions $\chi_j\colon K\to[0,1]$ ($j=0,1,\ldots,\ell$) as
\[
	\chi_j=
	\begin{cases}
	0 & \text{if\: $j=0$},\\
	\sum_{i=1}^j\rho_i & \text{if\: $j=1,\ldots,\ell$}.
	\end{cases}
\]
For a Hamiltonian $H\colon [0,1]\times M\to\mathbb{R}$ with $\supp(H)\subset [0,1]\times K$,
we define Hamiltonians $H^j$ ($j=0,1,\ldots,\ell$) and $L^j$ ($j=1,\ldots,\ell$) as
\[
	H^j(t,x)=\chi_j(x)H(t,x)
\]
and
\[
	L^j(t,x)=-H^{j-1}\bigl(t,\varphi_{H^{j-1}}^t(x)\bigr)+H^j\bigl(t,\varphi_{H^{j-1}}^t(x)\bigr)
\]
for $(t,x)\in [0,1]\times K$, respectively.
Since $\supp(H)\subset [0,1]\times K$, we can regard $H^j$ and $L^j$ as smooth functions on $[0,1]\times M$.
Fix $j=1,\ldots,\ell$.
Note that $L^j$ generates the Hamiltonian diffeomorphism $\varphi_{H^{j-1}}^{-1}\varphi_{H^j}$
and thus $\varphi_{H^j}=\varphi_{H^{j-1}}\varphi_{L^j}$.
Since $H^\ell=H$ and $H^0=0$,
\[
	\varphi_H=\varphi_{H^\ell}=\varphi_{H^{\ell-1}}\varphi_{L^{\ell}}=\cdots=\varphi_{H^0}\varphi_{L^1}\cdots\varphi_{L^\ell}=\varphi_{L^1}\cdots\varphi_{L^\ell}.
\]

Now, we claim that $\supp(L^{j})\subset [0,1]\times\overline{V'_j}$ if $H$ is $C^1$-small.
Since $H^{j-1}$ is also $C^1$-small,
$(\varphi_{H^{j-1}}^t)^{-1}(V_{j})\subset V'_j$.
Suppose that $x\notin\overline{V'_j}$.
Then, $(\varphi_{H^{j-1}}^t)(x)\notin\overline{V_j}$ and in particular, $(\varphi_{H^{j-1}}^t)(x)\notin\supp(\rho_j)$.
Since $\chi_j=\sum_{i=1}^{j}\rho_i$,
$\chi_{j-1}\bigl(\varphi_{H^{j-1}}^t(x)\bigr)=\chi_{j}\bigl(\varphi_{H^{j-1}}^t(x)\bigr)$ for any $t$.
Hence, $L^j(t,x)=0$ for any $x\notin \overline{V'_j}$ and any $t\in[0,1]$.
This completes the proof of the claim.

By $\supp(L^{j})\subset [0,1]\times\overline{V'_j}$ and the second condition on $\mathcal{V}'$,
$\supp(L^{j})\subset [0,1]\times U_j$.
Therefore, since $\varphi_H=\varphi_{L^1}\cdots\varphi_{L^\ell}$,
\[
	\|\varphi_H\|_\mathcal{U}\leq \ell.
\]
Thus, we can take $\ell$ as $N_{K,\mathcal{U}}$ in Lemma \ref{lemma:uniform}. 
\end{proof}


\section{Proof of the principal theorems}\label{section:pf of thm}

In this section, we prove Theorems \ref{theorem:main}, \ref{theorem:main2}, and \ref{theorem:main3}.

\subsection{Proof of Theorem \ref{theorem:main}}

\begin{proof}
For $i=1,\ldots,N$,
we choose a normalized time-independent Hamiltonian $H_i\colon M\to\mathbb{R}$ such that $H_i|_{X_i}\equiv 1$
and $X_j\cap\supp(H_i)=\emptyset$ whenever $j\in\{0,1,\ldots,N\}\setminus\{i\}$.
We define an injective homomorphism $I\colon\mathbb{R}^N\to\Ham(M)$ to be
\[
	I(r_1,\ldots,r_N)=\varphi_{r_1H_1+\cdots+r_N H_N}.
\]
Hence, it is enough to show that $I$ is bi-Lipschitz.

We fix $i=0,1,\ldots,N$ and set $r_0=0$.
Since $X_i$ is $c_i$-superheavy and $(r_1H_1+\cdots+r_NH_N)|_{X_i}\equiv r_i$,
Proposition \ref{proposition:shv} implies that
\begin{equation}\label{eq:shv}
	\sigma_{c_i}(\tilde{\varphi}_{r_1H_1+\cdots+r_N H_N})=r_i.
\end{equation}

We set $\alpha=\|\varphi_{r_1H_1+\cdots+r_N H_N}\|_U$.
Let us designate that
\[
	\varphi_{r_1H_1+\cdots+r_N H_N}=\varphi_{F_1}\cdots\varphi_{F_{\alpha}},
\]
where $F_{\ell}\in\mathcal{H}\bigl(\phi_{\ell}(U)\bigr)$ for some $\phi_{\ell}\in\Ham(M)$ for $\ell=1,\ldots,\alpha$.
Since $U$ satisfies the normally bounded spectrum condition with respect to $c_i$,
Proposition \ref{proposition:defect1} implies that there exists a positive number $K_i>0$ such that
\[
	\lvert\sigma_{c_i}(\tilde{\varphi}_{F_1}\cdots\tilde{\varphi}_{F_{\alpha}})-\sigma_{c_i}(\tilde{\varphi}_{F_1})-\sigma_{c_i}(\tilde{\varphi}_{F_2}\cdots\tilde{\varphi}_{F_{\alpha}})\rvert\leq K_i,
\]
\[
	\lvert\sigma_{c_i}(\tilde{\varphi}_{F_2}\cdots\tilde{\varphi}_{F_{\alpha}})-\sigma_{c_i}(\tilde{\varphi}_{F_2})-\sigma_{c_i}(\tilde{\varphi}_{F_3}\cdots\tilde{\varphi}_{F_{\alpha}})\rvert\leq K_i,
\]
\[
	\cdots
\]
\[
	\lvert\sigma_{c_i}(\tilde{\varphi}_{F_{\alpha-1}}\tilde{\varphi}_{F_{\alpha}})-\sigma_{c_i}(\tilde{\varphi}_{F_{\alpha-1}})-\sigma_{c_i}(\tilde{\varphi}_{F_{\alpha}})\rvert\leq K_i.
\]
Therefore, by Proposition \ref{proposition:VPNBSC} and the triangle inequality,
\begin{equation}\label{eq:mdefect}
	\left\lvert\sigma_{c_i}(\tilde{\varphi}_{F_1}\cdots\tilde{\varphi}_{F_{\alpha}})%
	+\sum_{\ell=1}^{\alpha}\langle F_{\ell}\rangle\right\rvert\leq(\alpha-1)K_i<\alpha K_i.
\end{equation}

Now, we define a map $\sigma'_i\colon\widetilde{\Ham}(M)\to\mathbb{R}$ to be $\sigma'_i=\sigma_{c_i}-\sigma_{c_0}$.
Then, by \eqref{eq:shv} and \eqref{eq:mdefect}, we obtain
\[
	\sigma'_i(\tilde{\varphi}_{r_1H_1+\cdots+r_N H_N})=r_i-r_0=r_i,
\]
and
\[
	\lvert\sigma'_i(\tilde{\varphi}_{F_1}\cdots\tilde{\varphi}_{F_{\alpha}})\rvert%
	<\alpha (K_i+K_0).
\]

Since the natural projection $\pi\colon\widetilde{\Ham}(M)\to\Ham(M)$ is a group homomorphism,
\[
	\pi(\tilde{\varphi}_{F_1}\cdots\tilde{\varphi}_{F_{\alpha}})%
	=\varphi_{F_1}\cdots\varphi_{F_{\alpha}}%
	=\varphi_{r_1H_1+\cdots+r_N H_N}.
\]
By assumption, $c_i$ descends asymptotically to $\Ham(M)$.
Hence, the map $\sigma'_i=\sigma_{c_i}-\sigma_{c_0}$ induces a map $\bar{\sigma}'_i\colon\Ham(M)\to\mathbb{R}$ such that $\sigma'_i=\bar{\sigma}'_i\circ\pi$.
Thus,
\begin{align*}
	\lvert r_i\rvert%
	&=\left\lvert\sigma'_i(\tilde{\varphi}_{r_1H_1+\cdots+r_N H_N})\right\rvert%
	=\lvert\bar{\sigma}'_i(\varphi_{F_1}\cdots\varphi_{F_{\alpha}})\rvert\\
	&=\lvert\sigma'_i(\tilde{\varphi}_{F_1}\cdots\tilde{\varphi}_{F_{\alpha}})\rvert%
	<\alpha(K_i+K_0).
\end{align*}
Hence,
\[
	\|\varphi_{r_1H_1+\cdots+r_N H_N}\|_U=\alpha%
	>(K_i+K_0)^{-1}\lvert r_i\rvert.
\]
Therefore,
\[
	\|\varphi_{r_1H_1+\cdots+r_N H_N}\|_U%
	>N^{-1}\left(\max_{1\leq i\leq N}K_i+K_0\right)^{-1}(\lvert r_1\rvert+\cdots+\lvert r_N\rvert).
\]

On the other hand, since $\supp(H_i)$ is compact for any $i$, by Lemma \ref{lemma:uniform},
there exist positive numbers $N_{i,U}$ and $\varepsilon$ such that
for any $0\leq t\leq\varepsilon$,
\[
	\|\varphi_{tH_i}\|_U=\|\varphi_{tH_i}\|_{\mathcal{U}_U}\leq N_{i,U}.
\]
Set $N_U=\max_i{N_{i,U}}$.
For each $i=1,\ldots,N$, choose a non-negative integer $a_i$ and a non-negative number $b_i$ with $b_i<\varepsilon$ such that
$r_i=a_i\varepsilon+b_i$.
Then,
\begin{align*}
	\|\varphi_{r_1H_1+\cdots+r_N H_N}\|_U%
	&\leq\|\varphi_{r_1H_1}\|_U+\cdots+\|\varphi_{r_NH_N}\|_U\\
	&\leq\|\varphi_{\varepsilon H_1}^{a_1}\varphi_{b_1H_1}\|_U+\cdots+\|\varphi_{\varepsilon H_N}^{a_N}\varphi_{b_NH_N}\|_U\\
	&\leq (a_1+\cdots+a_N+N)N_U.
\end{align*}
This completes the proof of Theorem \ref{theorem:main}.
\end{proof}


\subsection{Proof of Theorem \ref{theorem:main2}}

\begin{proof}
Since the proof is almost same as that of Theorem \ref{theorem:main},
we provide only the necessary changes.

Let $H_1,\ldots,H_N$ and $I\colon\mathbb{R}^N\to\Ham(M)$ be chosen as in the proof of Theorem \ref{theorem:main}.
We fix $i=0,1,\ldots,N$ and set $r_0=0$.
We set $\alpha=\|\varphi_{r_1H_1+\cdots+r_N H_N}\|_U$ and
\[
	\varphi_{r_1H_1+\cdots+r_N H_N}=(\phi_1^{-1}\varphi_{F_1}\phi_1)\cdots(\phi_{\alpha}^{-1}\varphi_{F_{\alpha}}\phi_{\alpha}),
\]
where $F_{\ell}\in\mathcal{H}(U)$ and $\phi_{\ell}\in\Ham(M)$ for $\ell=1,\ldots,\alpha$.
Since $U$ satisfies the asymptotically vanishing spectrum condition with respect to $c_i$,
Lemma \ref{sigma  is ci} implies that
\[
	\sigma_{c_i}(\tilde{\phi}_{\ell}^{-1}\tilde{\varphi}_{F_{\ell}}\tilde{\phi}_{\ell})%
	=\sigma_{c_i}(\tilde{\varphi}_{F_{\ell}})%
	=-\langle F_{\ell}\rangle
\]
for all $\ell$.
We define a map $\sigma'_i\colon\widetilde{\Ham}(M)\to\mathbb{R}$ to be $\sigma'_i=\sigma_{c_i}-\sigma_{c_0}$.
Then,
\begin{equation}\label{eq:AVP}
	\sigma'_i(\tilde{\phi}_{\ell}^{-1}\tilde{\varphi}_{F_{\ell}}\tilde{\phi}_{\ell})=-\langle F_{\ell}\rangle+\langle F_{\ell}\rangle=0
\end{equation}
for all $\ell$.

On the other hand, since the homogenization of a quasi-morphism is also a quasi-morphism (see, for example, \cite{Ca}),
there exists a positive number $K_i>0$ such that
\[
	\left\lvert\sigma_{c_i}(\tilde{\phi}\tilde{\psi})-\sigma_{c_i}(\tilde{\phi})-\sigma_{c_i}(\tilde{\psi})\right\rvert<K_i
\]
for any $\tilde{\phi},\tilde{\psi}\in\widetilde{\Ham}(M)$.
Hence, we obtain
\begin{equation}\label{eq:mdefect2}
	\left\lvert\sigma'_i(\tilde{\phi}\tilde{\psi})-\sigma'_i(\tilde{\phi})-\sigma'_i(\tilde{\psi})\right\rvert<K_i+K_0
\end{equation}
for any $\tilde{\phi},\tilde{\psi}\in\widetilde{\Ham}(M)$.
Using \eqref{eq:AVP} and \eqref{eq:mdefect2} several times yields
\[
	\left\lvert\sigma'_i\left((\tilde{\phi}_1^{-1}\tilde{\varphi}_{F_1}\tilde{\phi}_1)\cdots(\tilde{\phi}_{\alpha}^{-1}\tilde{\varphi}_{F_{\alpha}}\tilde{\phi}_{\alpha})\right)\right\rvert%
	<(\alpha-1)(K_i+K_0)<\alpha (K_i+K_0).
\]
Then, the remainder of the proof follows the same path as in Theorem \ref{theorem:main}.
\end{proof}


\subsection{Proof of Theorem \ref{theorem:main3}}

\begin{proof}
Let $H_1,\ldots,H_N$ and $I\colon\mathbb{R}^N\to\Ham(M)$ be chosen as in the proof of Theorem \ref{theorem:main}.
We fix $i=1,\ldots,N$ and set $r_0=0$.
We define a map $\sigma'_i\colon\widetilde{\Ham}(M)\to\mathbb{R}$ to be $\sigma'_i=\sigma_{c_i}-\sigma_{c_0}$.
Then, by \eqref{eq:shv},
\[
	\sigma'_i(\widetilde{\varphi}_{r_1H_1+\cdots+r_N H_N})=r_i-r_0=r_i.
\]

We set $\alpha=\|\varphi_{r_1H_1+\cdots+r_N H_N}\|_{\mathcal{U}}$.
Let us denote
\[
	\varphi_{r_1H_1+\cdots+r_N H_N}=\varphi_{F_1}\cdots\varphi_{F_{\alpha}},
\]
where $F_{\ell}\in\mathcal{H}(U_{\lambda_{\ell}})$ for some $U_{\lambda_{\ell}}\in\mathcal{U}$.

Since $U_{\lambda_{\ell}}$ satisfies the bounded spectrum condition with respect to $ c_i$ and $ c_0$ for all $\ell$,
there exist positive numbers $K_i,K_0>0$ such that
\begin{equation}\label{eq:bdd}
	 c_i(\tilde{\varphi}_{F_{\ell}})+\langle F_{\ell}\rangle<K_i%
	\quad\text{and}\quad%
	 c_0(\tilde{\varphi}_{F_{\ell}}^{-1})-\langle F_{\ell}\rangle<K_0
\end{equation}
for all $\ell$.
We claim that
\[
	\left\lvert\sigma'_i(\tilde{\varphi}_{F_1}\cdots\tilde{\varphi}_{F_{\alpha}})\right\rvert%
	<\alpha(K_i+K_0).
\]
Indeed, by subadditivity,
\[
	 c_i\left((\tilde{\varphi}_{F_1}\cdots\tilde{\varphi}_{F_{\alpha}})^k\right)\leq k c_i(\tilde{\varphi}_{F_1})+\cdots+k c_i(\tilde{\varphi}_{F_{\alpha}}),
\]
\[
	 c_0\left((\tilde{\varphi}_{F_{\alpha}}^{-1}\cdots\tilde{\varphi}_{F_1}^{-1})^k\right)\leq k c_0(\tilde{\varphi}_{F_1}^{-1})+\cdots+k c_0(\tilde{\varphi}_{F_{\alpha}}^{-1}),
\]
and
\[
	- c_0\left((\tilde{\varphi}_{F_1}\cdots\tilde{\varphi}_{F_{\alpha}})^k\right)\leq c_0\left((\tilde{\varphi}_{F_{\alpha}}^{-1}\cdots\tilde{\varphi}_{F_1}^{-1})^k\right)-c(\mathbbm{1}).
\]
By combining with \eqref{eq:bdd}, we obtain
\begin{align*}
	 c_i\left((\tilde{\varphi}_{F_1}\cdots\tilde{\varphi}_{F_{\alpha}})^k\right)- c_0\left((\tilde{\varphi}_{F_1}\cdots\tilde{\varphi}_{F_{\alpha}})^k\right)%
	&\leq k\sum_{\ell=1}^{\alpha}\left( c_i(\tilde{\varphi}_{F_{\ell}})+ c_0(\tilde{\varphi}_{F_{\ell}}^{-1})\right)-c(\mathbbm{1})\\
	&<k\alpha(K_i+K_0)-c(\mathbbm{1}).
\end{align*}
Similarly,
\[
	 c_0\left((\tilde{\varphi}_{F_1}\cdots\tilde{\varphi}_{F_{\alpha}})^k\right)- c_i\left((\tilde{\varphi}_{F_1}\cdots\tilde{\varphi}_{F_{\alpha}})^k\right)
	<k\alpha(K_i+K_0)-c(\mathbbm{1}).
\]
Therefore,\[
	\left\lvert c_i\left((\tilde{\varphi}_{F_1}\cdots\tilde{\varphi}_{F_{\alpha}})^k\right)- c_0\left((\tilde{\varphi}_{F_1}\cdots\tilde{\varphi}_{F_{\alpha}})^k\right)\right\rvert%
	<\lvert k\alpha(K_i+K_0)-c(\mathbbm{1})\rvert.
\]
Thus, dividing by $k$ and passing to the limit as $k\to\infty$ yields
\[
	\lvert\sigma'_i(\tilde{\varphi}_{F_1}\cdots\tilde{\varphi}_{F_{\alpha}})\rvert%
	=\lim_{k\to\infty}\frac{\left\lvert c_i\left((\tilde{\varphi}_{F_1}\cdots\tilde{\varphi}_{F_{\alpha}})^k\right)- c_0\left((\tilde{\varphi}_{F_1}\cdots\tilde{\varphi}_{F_{\alpha}})^k\right)\right\rvert}k%
	<\alpha(K_i+K_0).
\]

Since the natural projection $\pi\colon\widetilde{\Ham}(M)\to\Ham(M)$ is a group homomorphism,
\[
	\pi(\tilde{\varphi}_{F_1}\cdots\tilde{\varphi}_{F_{\alpha}})=\varphi_{F_1}\cdots\varphi_{F_{\alpha}}=\varphi_{r_1H_1+\cdots+r_N H_N}.
\]
By assumption, $c_i$ and $c_0$ descend asymptotically to $\Ham(M)$.
Hence, the map $\sigma'_i=\sigma_{c_i}-\sigma_{c_0}$ induces a map $\bar{\sigma}'_i\colon\Ham(M)\to\mathbb{R}$ such that $\sigma'_i=\bar{\sigma}'_i\circ\pi$.
Thus,
\[
	\lvert r_i\rvert=\lvert\sigma'_i(\tilde{\varphi}_{r_1H_1+\cdots+r_N H_N})\rvert%
	=\lvert\bar{\sigma}'_i(\varphi_{F_1}\cdots\varphi_{F_{\alpha}})\rvert%
	=\lvert\sigma'_i(\tilde{\varphi}_{F_1}\cdots\tilde{\varphi}_{F_{\alpha}})\rvert%
	<\alpha(K_i+K_0).
\]
Then, the remainder of the proof follows the same path as in Theorem \ref{theorem:main}.
\end{proof}


\section{Lagrangian spectral invariants}\label{section:lagspecinv}

Lagrangian spectral invariants for monotone Lagrangian submanifolds were defined by Leclercq and Zapolsky \cite{LZ}.
In this section, we review their properties and prove the corollaries given in Section \ref{section:app}.

Let $(M,\omega)$ be a closed symplectic manifold.
Let $L$ be a monotone Lagrangian submanifold of $(M,\omega)$ with minimal Maslov number $N_L\geq 2$
(For the definitions of the monotonicity and the minimal Maslov number of a Lagrangian submanifold, see \cite{Oh96,BC,LZ} for example).

We fix a commutative ring $R$. 
Assuming that $L$ is relatively $\mathrm{Pin}^{\pm}$ (see \cite[Section 7.1]{Za} for the definition),
Zapolsky defined the \textit{Lagrangian quantum homology}
\footnote{
In Leclercq and Zapolsky's terminology, our Lagrangian quantum homology $\QH_{\ast}(L;R)$ (resp., Lagrangian Floer homology $\HF_{\ast}(L;R)$)
is the \textit{quotient} Lagrangian quantum homology $\QH_{\ast}^{\pi_2^0(M,L)}(L;R)$ (resp., \textit{quotient} Lagrangian Floer homology $\HF_{\ast}^{\pi_2^0(M,L)}(L;R)$),
where $\pi_2^0(M,L)$ is the kernel of $[\omega]\colon \pi_2(M,L)\to\mathbb{R}$.
}
$\QH_{\ast}(L;R)$ of $L$ \cite[Sections 4 and 7.3]{Za}.
Moreover, he defined the \textit{Lagrangian Floer homology} $\HF_{\ast}(L;R)$ of $L$
and proved that there exists an isomorphism $\QH_{\ast}(L;R)\cong\HF_{\ast}(L;R)$
called the \textit{Piunikhin--Salamon--Schwarz isomorphism} \cite[Sections 5 and 7.3]{Za}.
His work generalizes that of Oh and that of Biran and Cornea \cite{Oh96,BC}.

We assume that $\QH_{\ast}(L;R)\neq0$.
Following \cite[Section 3]{LZ}, one can define the Lagrangian spectral invariant $c^{(L;R)}\colon\widetilde{\Ham}(M)\to\mathbb{R}$
associated with the fundamental class $[L]\in\QH_{\ast}(L;R)$.
Moreover, Leclercq and Zapolsky proved that $c^{(L;R)}$ is a subadditive invariant \cite[Theorem 41]{LZ}.

To prove Corollaries \ref{cp2}, \ref{surface main theorem}, \ref{annulus frag} and Theorem \ref{calabi on surface}, we first prove the following result.

\begin{theorem}\label{Lagrangian descending}
If $(M,\omega)$ is either $(\mathbb{C}P^n,\omega_{\mathrm{FS}})$ or $(\Sigma_{\vec{g}},\omega)$,
then the Lagrangian spectral invariant $c^{(L;R)}\colon\widetilde{\Ham}(M)\to\mathbb{R}$ descends asymptotically to $\Ham(M)$.
\end{theorem}

One can also define the quantum homology
\footnote{
Similar to the above, our quantum homology $\QH_{\ast}(M;R)$
is Zapolsky's \textit{quotient} quantum homology $\QH_{\ast}^{\pi_2^0(M)}(M;R)$,
where $\pi_2^0(M)$ is the kernel of $[\omega]\colon \pi_2(M)\to\mathbb{R}$.
}
$\QH_{\ast}(M;R)$ of the ambient manifold $(M,\omega)$ \cite[Sections 4.5 and 7.2]{Za}.
Let $c^{(M;R)}\colon\widetilde{\Ham}(M)\to\mathbb{R}$ be the Oh--Schwarz spectral invariant associated with the fundamental class $[M]\in \QH_{\ast}(M;R)$.
$c^{(M;R)}$ is also a subadditive invariant (see, for example, \cite[Theorem I]{Oh05}).

Now, we have the \textit{quantum module action}
\[
	\bullet\colon\QH_{\ast}(M;R)\otimes\QH_{\ast}(L;R)\to\QH_{\ast}(L;R)
\]
(see \cite[Section 7.4]{Za}, \cite[Section 2.5.3]{LZ}).
\cite[Proposition 5]{LZ} then yields the following inequality as a corollary.


\begin{proposition}[{\cite[Proposition 5]{LZ}}]\label{ham and lagr}
For any Hamiltonian $H\colon S^1\times M\to\mathbb{R}$,
\[
	c^{(L;R)}(\tilde{\varphi}_H)\leq c^{(M;R)}(\tilde{\varphi}_H).
\]
\end{proposition}

\begin{proof}[Proof of Theorem \ref{Lagrangian descending}]
As a consequence of Schwarz \cite{Sch}, $c^{(\Sigma_{\vec{g}};R)}$ descends asymptotically to $\Ham(\Sigma_{\vec{g}})$.
As a consequence of Entov and Polterovich \cite{EP03}, $c^{(\mathbb{C}P^n;R)}$ descends asymptotically to $\Ham(\mathbb{C}P^n)$.
Thus, Theorem \ref{Lagrangian descending} follows from Propositions \ref{comparing descending} and \ref{ham and lagr}.
\end{proof}

When $c^{(M;R)}\colon\widetilde{\Ham}(M)\to\mathbb{R}$ is a quasi-morphism,
Proposition \ref{ham and lagr} enables us to prove the following proposition.

\begin{proposition}\label{quasiquasi}
If $c^{(M;R)}$ is a quasi-morphism, then $c^{(L;R)}$ is as well.
\end{proposition}

\begin{proof}
For the sake of brevity, we write $c^L=c^{(L;R)}$ and $c^M=c^{(M;R)}$.
By subadditivity,
\[
	c^L(\tilde{\phi}\tilde{\psi})-c^L(\tilde{\phi})-c^L(\tilde{\psi})\leq 0
\]
for any $\tilde{\phi},\tilde{\psi}\in\widetilde{\Ham}(M)$.
Hence, it is sufficient to show that there exists a positive number $K>0$ such that
\[
	c^L(\tilde{\phi}\tilde{\psi})-c^L(\tilde{\phi})-c^L(\tilde{\psi})>-K
\]
for any $\tilde{\phi},\tilde{\psi}\in\widetilde{\Ham}(M)$.

Since $c^M$ is a quasi-morphism, there exists a positive number $C>0$ such that
\[
	c^M(\mathbbm{1})-c^M(\tilde{\psi})-c^M(\tilde{\psi}^{-1})>-C.
\]
Then, subadditivity and Proposition \ref{ham and lagr} imply that
\begin{align*}
	c^L(\tilde{\phi}\tilde{\psi})-c^L(\tilde{\phi})-c^L(\tilde{\psi})%
	&\geq -c^L(\tilde{\psi})-c^L(\tilde{\psi}^{-1})\\
	&\geq -c^M(\tilde{\psi})-c^M(\tilde{\psi}^{-1})%
	>-c^M(\mathbbm{1})-C.\qedhere
\end{align*}
\end{proof}

To prove Corollaries \ref{cp2}, \ref{surface main theorem}, and \ref{annulus frag} and Theorem \ref{calabi on surface},
we use the following propositions.

\begin{proposition}[\cite{LZ,Ka2}]\label{L is shv}
$L$ is $c^{(L;R)}$-superheavy.
\end{proposition}

\begin{proposition}[\cite{Ka2}]\label{relative disp}
Any open subset $U\subset M$ displaceable from $L$ satisfies the bounded spectrum condition with respect to $c^{(L;R)}$.
\end{proposition}

\begin{proposition}[{\cite[Proposition 3.1]{U}}, \cite{Ka2}]\label{displaceable SC}
Any abstractly displaceable open subset $U\subset M$ satisfies the normally bounded spectrum condition with respect to $c^{(M;R)}$ and $c^{(L;R)}$.
\end{proposition}


\section{Proof of corollaries}\label{section:pf of cor}

In this section, we prove Corollaries \ref{s2s2}, \ref{cp2}, \ref{surface main theorem}, and \ref{annulus frag}.
For the sake of brevity, let $\mathbb{Z}_2$ denote the field $\mathbb{Z}/2\mathbb{Z}$ below.

\subsection{Proof of Corollary \ref{b2n}}

We think of the ball $B^{2n}$ as embedded in $\mathbb{C}^n$
and consider the mutually disjoint tori
\[
	T_{\delta}=\left\{\,(w_1,\ldots,w_n)\in\mathbb{C}^n\relmiddle|\lvert w_i\rvert^2=\frac{1}{\delta(n+1)}\ \text{for any}\ i=1,\ldots,n\,\right\},
\]
$0<\delta\leq 1$, where $w_1,\ldots,w_n$ are the standard complex coordinates on $\mathbb{C}^n$.
Let $(\mathbb{C}P^n,\omega_{\mathrm{FS}})$ be $n$-dimensional complex projective space
and $L_C=\{\lvert z_0\rvert=\cdots=\lvert z_n\rvert\}\subset\mathbb{C}P^n$ the Clifford torus.
For a positive number $\delta$ with $\delta\in(\frac{n}{n+1},1]$, Biran, Entov, and Polterovich \cite[Section 4]{BEP} constructed a conformally symplectic embedding
$\vartheta_{\delta}\colon B^{2n}\to\mathbb{C}P^n$
satisfying $\vartheta_{\delta}(T_{\delta})=L_C$.
The embeddings $\vartheta_{\delta}\colon B^{2n}\to\mathbb{C}P^n$, $\delta\in(\frac{n}{n+1},1]$, induce
homomorphisms $(\vartheta_{\delta})_{\ast}\colon\widetilde{\Ham}(B^{2n})\to\widetilde{\Ham}(\mathbb{C}P^n)$.

Let $c^{\mathbb{C}P^n}=c^{(\mathbb{C}P^n;\mathbb{Z}_2)}\colon\widetilde{\Ham}(\mathbb{C}P^n)\to\mathbb{R}$ be the Oh--Schwarz spectral invariant
associated with $[\mathbb{C}P^n]\in\QH_{\ast}(\mathbb{C}P^n;\mathbb{Z}_2)$.
According to \cite[Theorem 3.1]{EP03}, $c^{\mathbb{C}P^n}$ is a quasi-morphism.
Therefore, the functions $c'_{\delta}=c^{\mathbb{C}P^n}\circ(\vartheta_{\delta})_{\ast}\colon\widetilde{\Ham}(B^{2n})\to\mathbb{R}$,
$\delta\in(\frac{n}{n+1},1]$, are subadditive invariants and quasi-morphisms.

Biran, Entov, and Polterovich proved that there exists a constant $c_n$ such that $T_{\delta}$ is $c_{\delta}$-superheavy,
where $c_{\delta}=c_n\cdot c'_{\delta}$
for any $\delta\in(\frac{n}{n+1},1]$.
Since $T_\delta\cap B(r)=\emptyset$ holds for any $\delta$ with $\delta\in(\frac{n}{n+1},(\frac{n}{n+1})\cdot r^{-2})$, by Proposition \ref{disp and vanish}, the open ball $B(r)$ of radius $r<1$ satisfies the asymptotically vanishing condition
with respect to $c_{\delta}$ for any $\delta\in(\frac{n}{n+1},(\frac{n}{n+1})\cdot r^{-2})$.
Thus, Theorem \ref{theorem:main2} completes the proof of Corollary \ref{b2n}.


\subsection{Proof of Corollary \ref{s2s2}}

First we recall the definition of stems.
Let $(M,\omega)$ be a closed symplectic manifold.
Let $\mathbb{A}$ be a finite-dimensional Poisson-commutative subspace of $C^{\infty}(M)$.
Let $\Phi\colon M\to\mathbb{A}^{\ast}$ be the moment map given by $F(x)=\langle\Phi(x),F\rangle$ for $x\in M$ and $F\in\mathbb{A}$.

\begin{definition}[{\cite[Definition 2.3]{EP06}}]
A closed subset $X$ of $M$ is called a \textit{stem}
if there exists a finite-dimensional Poisson-commutative subspace $\mathbb{A}$ of $C^{\infty}(M)$
such that $X$ is a fiber of $\Phi$ and each non-trivial fiber of $\Phi$, other than $X$, is displaceable.
\end{definition}

The proof of the following theorem is quite similar to that of \cite[Theorem 1.8]{EP09}.

\begin{theorem}[{\cite[Theorem 1.8]{EP09}}]
Every stem is $c$-superheavy,
where $c$ is a Lagrangian spectral invariant defined in \cite{LZ} or a spectral invariant defined in \cite{FOOO}.
\end{theorem}

\begin{proof}[Proof of Corollary \ref{s2s2}]
Fukaya, Oh, Ohta, and Ono \cite{FOOO} defined a family of bulk-deformed Oh--Schwarz spectral invariants $\{c_\rho\}_{\rho\in[0,1/2)}$ on $\widetilde{\Ham}(S^2\times S^2)$
and proved that any $c_\rho$ descends asymptotically to $\Ham(S^2\times S^2)$.
They also constructed a family of mutually disjoint Lagrangian submanifolds $T(\rho)$ ($\rho\in[0,1/2)$) and proved that each $T(\rho)$ is $c_{\rho}$-superheavy \cite[Theorem 23.4]{FOOO}.
It is known that when $U$ is abstractly displaceable,
$U$ satisfies the bounded spectrum condition with respect to $c_{\rho}$ for any $\rho$ (see also Proposition \ref{displaceable SC}).

On the other hand, $E\times E\subset S^2\times S^2$ is a stem.
In particular, $E\times E$ is $c_{\rho}$-superheavy for any $\rho$.
Hence, Proposition \ref{disp and vanish} implies that
if $U$ is displaceable from $E\times E$,
then $U$ satisfies the asymptotically vanishing spectrum condition with respect to $c_\rho$ for any $\rho$.

Therefore, in any case,
$U$ satisfies the asymptotically vanishing spectrum condition with respect to $c_\rho$ for any $\rho$ (see also Remark \ref{remark:AVSC}).
Since $c_\rho$ is known to be a quasi-morphism for any $\rho$, 
Theorem \ref{theorem:main2} completes the proof of Corollary \ref{s2s2}.
\end{proof}


\subsection{Proof of Corollary \ref{cp2}}

\begin{proof}
By Biran and Cornea's work \cite[Corollary 1.2.11 (ii)]{BC},
$\QH_{\ast}(\mathbb{R}P^2;\mathbb{Z}_2)\cong\HF_{\ast}(\mathbb{R}P^2;\mathbb{Z}_2)\cong\mathbb{Z}_2$
(see also \cite[Section 2.6.1]{LZ})
\footnote{
Given a Lagrangian submanifold $L$ of a symplectic manifold $(M,\omega)$,
our Lagrangian quantum homology $\QH_{\ast}(L;\mathbb{Z}_2)$
is actually Biran and Cornea's $\QH_{\ast}(L;\Lambda)$ where $\Lambda=\mathbb{Z}_2[t,t^{-1}]$.
}.
Moreover, Leclercq and Zapolsky \cite[Section 2.6.3]{LZ} showed that
$\QH_{\ast}(L_W;\mathbb{Z})\neq 0$.
Let $c^{\mathbb{R}P^2}=c^{(\mathbb{R}P^2;\mathbb{Z}_2)}$ and $c^{L_W}=c^{(L_W;\mathbb{Z})}$.

By Theorem \ref{Lagrangian descending},
$c^{\mathbb{R}P^2}$ and $c^{L_W}$ descend asymptotically to $\Ham(\mathbb{C}P^2)$.
According to \cite[Theorem 3.1]{EP03},
the Oh--Schwarz spectral invariant $c^{(\mathbb{C}P^2;R)}$ is a quasi-morphism for
$R=\mathbb{Z}_2$ and $\mathbb{Z}$.
Hence, by Proposition \ref{quasiquasi}, $c^{\mathbb{R}P^2}$ and $c^{L_W}$ are also quasi-morphisms.
Moreover, by Proposition \ref{L is shv}, $\mathbb{R}P^2$ and $L_W$ are superheavy with respect to
$c^{\mathbb{R}P^2}$ and $c^{L_W}$, respectively.

When $U$ is abstractly displaceable (case (i)),
Proposition \ref{displaceable SC} ensures that
$U$ satisfies the normally bounded spectrum condition with respect to $c^{\mathbb{R}P^2}$ and $c^{L_W}$.

When $U$ is displaceable from $\mathbb{R}P^2$ and $L_W$ (case (ii)),
Proposition \ref{relative disp} ensures that
$U$ satisfies the bounded spectrum condition with respect to $c^{\mathbb{R}P^2}$ and $c^{L_W}$.

In addition, the Clifford torus $L_C$ is a stem \cite{BEP}.
In particular, $L_C$ is superheavy with respect to $c^{\mathbb{R}P^2}$ and $c^{L_W}$.
Hence, Proposition \ref{disp and vanish} implies that
if $U$ is displaceable from $L_C$ (case (iii)),
then $U$ satisfies the asymptotically vanishing spectrum condition with respect to $c^{\mathbb{R}P^2}$ and $c^{L_W}$.

Therefore, in any case,
$U$ satisfies the asymptotically vanishing spectrum condition with respect to $c^{\mathbb{R}P^2}$ and $c^{L_W}$.
Since $\mathbb{R}P^2\cap L_W=\emptyset$, we conclude that
$c^{\mathbb{R}P^2}$, $c^{L_W}$, $\mathbb{R}P^2$, $L_W$ and $U$
satisfy the assumption of Theorem \ref{theorem:main2} for $N=1$.
This completes the proof of Corollary \ref{cp2}.
\end{proof}

\begin{remark}
We do not need Theorem \ref{Lagrangian descending} to prove Corollary \ref{cp2} if we use the well-known fact that $\pi_1\bigl(\Ham(\mathbb{C}P^2)\bigr)=0$ (see \cite{G}).
We provide a more general argument here for future works.
\end{remark}


\subsection{Proof of Corollary \ref{surface main theorem}}

We use the following result to prove Corollary \ref{surface main theorem}.

\begin{proposition}[\cite{Po}, \cite{Ka}, {\cite[Proposition 4.4]{Is}}, {\cite[Theorem 1.9]{Z}}]\label{master thesis}
For any positive integer $g$,
there exists a positive number $K$ such that
\[
	c^{(\Sigma_g;\mathbb{Z}_2)}(\tilde{\varphi}_F)+\langle F\rangle\leq K
\]
for any contractible open subset $U$ of $\Sigma_g$
and any $F\in\mathcal{H}(U)$.
\end{proposition}

\begin{proof}[Proof of Corollary \ref{surface main theorem}]
Let $C$  be a non-contractible simple closed curve in the surface $\Sigma_g$.
We choose symplectomorphisms $f_1,\ldots,f_N$ of $(\Sigma_g,\omega)$ to ensure that the subsets
$C$, $f_1(C),\ldots,f_N(C)$ are mutually disjoint.
We fix $i=0,1,\ldots,N$.
We set $L_i=f_i(C)$, where $f_0=\mathrm{id}_{\Sigma_g}$.
Then, $\QH_{\ast}(L_i;\mathbb{Z}_2)\cong\HF_{\ast}(L_i;\mathbb{Z}_2)$ does not vanish.
Let $c^{L_i}=c^{(L_i;\mathbb{Z}_2)}$ denote the associated Lagrangian spectral invariant.
Remark \ref{MVZ}, Propositions \ref{master thesis} and \ref{ham and lagr} imply that
$U$ satisfies the normally bounded spectrum condition with respect to $c^{L_i}$ for any $i$.

Then, Theorem \ref{Lagrangian descending} and Proposition \ref{L is shv} ensure that
$c^{L_0},\ldots,c^{L_N}$,
$L_0,\ldots,L_N$ and $U$
satisfy the assumption of Theorem \ref{theorem:main}.
This completes the proof of Corollary \ref{surface main theorem}.
\end{proof}

\begin{remark}
We do not need Theorem \ref{Lagrangian descending} to prove Corollary \ref{surface main theorem} if we use the well-known fact that $\pi_1\bigl(\Ham(\Sigma_g)\bigr)=0$ for positive $g$ (see \cite[Section 7.2.B]{Po01}).
We provide a more general argument here for future works.
\end{remark}


\subsection{Proof of Corollary \ref{annulus frag}}

\begin{proof}
In the proof of Corollary \ref{surface main theorem},
we constructed mutually disjoint Lagrangian submanifolds $L_0,\ldots,L_N\subset(\Sigma_g,\omega)$
such that $\QH_{\ast}(L_i;\mathbb{Z}_2)$ does not vanish.
By the construction of $L_0,\ldots,L_N$ and the assumption on the covering $\mathcal{U}$,
each $U_{\lambda}$ is displaceable from $L_i$.

Then, Theorem \ref{Lagrangian descending} and Propositions \ref{L is shv} and \ref{relative disp} ensure that
$c^{L_0},\ldots,c^{L_N}$,
$L_0,\ldots,L_N$ and $\mathcal{U}=\{U_\lambda\}_{\lambda}$
satisfy the assumption of Theorem \ref{theorem:main3}.
This completes the proof of Corollary \ref{annulus frag}.
\end{proof}


\section{Partial Calabi quasi-morphisms}\label{section:calabi}

Let $(M,\omega)$ be a closed symplectic manifold.
Given an open subset $U\subset M$ such that $\omega|_U$ is exact, we recall that the \textit{Calabi homomorphism}
is a homomorphism $\mathrm{Cal}_U\colon\Ham(U)\to\mathbb{R}$ defined by
\[
	\mathrm{Cal}_U(\varphi_F)=\int_0^1\int_U F_t\omega^n\,dt.
\]

\begin{definition}[\cite{E}]\label{definition of Calabi}
A \textit{partial Calabi quasi-morphism} is a function $\mu\colon\Ham(M)\to\mathbb{R}$ satisfying the following conditions.
\begin{description}
	\item[Stability] For any Hamiltonians $H,K\colon S^1\times M\to\mathbb{R}$,
	\[
		\int_0^1\min_M(H_t-K_t)\,dt\leq\frac{\mu(\varphi_H)-\mu(\varphi_K)}{\Vol(M)}\leq\int_0^1\max_M(H_t-K_t)\,dt.
	\]
	\item[Partial homogeneity] $\mu(\phi^k)=k\mu(\phi)$ for any $\phi\in\Ham(M)$ and $k\in\mathbb{Z}_{\geq 0}$.
	\item[Partial quasi-additivity] Given a displaceable open subset $U\subset M$, there exists a positive number $K>0$ such that
	\[
		\lvert\mu(\phi\psi)-\mu(\phi)-\mu(\psi)\rvert\leq K\min\{\|\phi\|_U,\|\psi\|_U\}
	\]
	for any $\phi,\psi\in\Ham(M)$.
	\item[Calabi property] For any displaceable open subset $U\subset M$ such that $\omega|_U$ is exact,
	the restriction of $\mu$ to $\Ham(U)$ coincides with the Calabi homomorphism $\mathrm{Cal}_U$.
\end{description}
\end{definition}

Let $c\colon\widetilde{\Ham}(M)\to\mathbb{R}$ be a subadditive invariant descending asymptotically to $\Ham(M)$ (see Definition \ref{definition:descend}).
Let $U$ be an open subset of $M$ satisfying the normally bounded spectrum condition with respect to $c$ (see Definition \ref{definition:NBSC}).
We can generalize Proposition \ref{proposition:defect1} as follows.

\begin{proposition}\label{proposition:defect}
There exists a positive number $K>0$ such that
\[
	\lvert\bar{\sigma}_c(\phi\psi)-\bar{\sigma}_c(\phi)-\bar{\sigma}_c(\psi)\rvert\leq K\min\{\|\phi\|_U,\|\psi\|_U\}
\]
for any $\phi,\psi\in\Ham(M)$.
\end{proposition}

\begin{proof}
We assume, without loss of generality, that $\|\phi\|_U\leq\|\psi\|_U$.
We represent $\phi\in\Ham(M)$ as $\phi=\phi_1\cdots\phi_{\alpha}$ with $\|\phi_i\|_U=1$ for all $i$.
We claim that
\[
	\lvert\bar{\sigma}_c(\phi\psi)-\bar{\sigma}_c(\phi)-\bar{\sigma}_c(\psi)\rvert\leq C(2\alpha-1)
\]
for some $C>0$.
Then, the proposition follows if we set $K=2C$.
We prove the claim by induction on $\alpha=\|\phi\|_U$.

When $\alpha=1$, we can choose a Hamiltonian $F$ such that $\varphi_F=\phi$ and $F\in\mathcal{H}\bigl(\theta(U)\bigr)$ for some $\theta\in\Ham(M)$.
Then, Proposition \ref{proposition:defect1} implies that
\[
	\lvert\bar{\sigma}_c(\phi\psi)-\bar{\sigma}_c(\phi)-\bar{\sigma}_c(\psi)\rvert%
	=\lvert\bar{\sigma}_c(\varphi_F\psi)-\bar{\sigma}_c(\varphi_F)-\bar{\sigma}_c(\psi)\rvert%
	\leq C
\]
for some $C>0$.
This proves the claim.

We assume that the claim holds for $\|\phi\|_U=\alpha$.
For $\phi\in\Ham(M)$ with $\|\phi\|_U=\alpha+1$, we decompose it into
$\phi=\phi_{\alpha}\phi_1$ where $\|\phi_{\alpha}\|_U=\alpha$ and $\|\phi_1\|_U=1$.
By the induction hypothesis,
\[
	\lvert\bar{\sigma}_c(\phi\psi)-\bar{\sigma}_c(\phi_{\alpha})-\bar{\sigma}_c(\phi_1\psi)\rvert\leq C(2\alpha-1).
\]
Moreover, since $\|\phi_1\|_U=1$,
\[
	\lvert\bar{\sigma}_c(\phi_1\psi)-\bar{\sigma}_c(\phi_1)-\bar{\sigma}_c(\psi)\rvert\leq C
	\quad\text{and}\quad
	\lvert\bar{\sigma}_c(\phi_1)+\bar{\sigma}_c(\phi_{\alpha})-\bar{\sigma}_c(\phi)\rvert\leq C.
\]
Hence,
\[
	\lvert\bar{\sigma}_c(\phi\psi)-\bar{\sigma}_c(\phi)-\bar{\sigma}_c(\psi)\rvert\leq C(2\alpha+1).
\]
This completes the proof of Proposition \ref{proposition:defect}.
\end{proof}

\begin{remark}
Since the fragmentation norm $\|\cdot\|_{\mathcal{U}}$ with respect to a covering $\mathcal{U}$ is not conjugation invariant in general,
we cannot prove a proposition corresponding to Proposition \ref{proposition:defect} in the same manner
(see also the proof of Proposition \ref{proposition:defect1}).
\end{remark}


\subsection{Proof of Theorem \ref{calabi on surface}}

For $\vec{g}=(g_1,\ldots,g_n)\in\mathbb{N}^n$,
we recall that $(\Sigma_{\vec{g}},\omega)$ is the product manifold $\Sigma_{\vec{g}}=\Sigma_{g_1}\times\cdots\times\Sigma_{g_n}$
equipped with a symplectic form $\omega$.

\begin{proof}
Let $C_i$  be a non-contractible simple closed curve in the surface $\Sigma_{g_i}$,
and let $C$ denote the Lagrangian submanifold $C_1\times\cdots\times C_n$ of $(\Sigma_{\vec{g}},\omega_{\vec{A}})$.

For all positive integers $N$, we choose symplectomorphisms $f_1,\ldots,f_N$ of $(\Sigma_{\vec{g}},\omega_{\vec{A}})$ to ensure that the subsets
$C$, $f_1(C),\ldots,f_N(C)$ are mutually disjoint.
We fix $i=0,1,\ldots,N$.
We set $L_i=f_i(C)$, where $f_0=\mathrm{id}_{\Sigma_{\vec{g}}}$.
Then, $\QH_{\ast}(L_i;\mathbb{Z}_2)\cong\HF_{\ast}(L_i;\mathbb{Z}_2)$ does not vanish.
Let $c_i=c^{(L_i;\mathbb{Z}_2)}\colon\widetilde{\Ham}(\Sigma_{\vec{g}})\to\mathbb{R}$
denote the Lagrangian spectral invariant associated with $[L_i]\in\QH_{\ast}(L_i;\mathbb{Z}_2)$.
By Theorem \ref{Lagrangian descending}, $c_i$  descends asymptotically to $\Ham(\Sigma_{\vec{g}})$.

Now, we define a function $\mu_i\colon\Ham(\Sigma_{\vec{g}})\to\mathbb{R}$
by $\mu_i=-\Vol(\Sigma_{\vec{g}})\cdot\bar{\sigma}_{c_i}$.
By definition, $\mu_i$ satisfies partial homogeneity.
By Proposition \ref{displaceable SC},
any displaceable open subset of $\Sigma_{\vec{g}}$ satisfies the normally bounded spectrum condition with respect to $c_i$.
Hence, Proposition \ref{proposition:defect} implies partial quasi-additivity.
Moreover, the Calabi property follows from Proposition \ref{proposition:VPNBSC}.
In fact, for any displaceable open subset $U$ such that $\omega|_U$ is exact,
and any Hamiltonian $F\in\mathcal{H}(U)$,
\[
	\mu_i(\varphi_F)=-\Vol(\Sigma_{\vec{g}})\cdot\bar{\sigma}_{c_i}(\varphi_F)%
	=\Vol(\Sigma_{\vec{g}})\cdot\langle F\rangle=\mathrm{Cal}_U(\varphi_F).
\]
Finally, \cite[Theorem 41]{LZ} ensures the stability of $\mu_i$.
Hence, $\mu_i$ is a partial Calabi quasi-morphism.

By construction, $\mu_0,\mu_1,\ldots,\mu_N$ are linearly independent.
This completes the proof of Theorem \ref{calabi on surface}.
\end{proof}


\section{Problems}

The authors are yet to find the answers to the following problems.

\begin{problem}
Let $(M,\omega)$ be a closed symplectic manifold.
Let $c\colon\widetilde{\Ham}(M)\to\mathbb{R}$ be either the Oh--Schwarz spectral invariant or the Lagrangian spectral invariant defined in \cite{LZ}.
Does there exist an open subset $U$ of $M$ satisfying the asymptotically vanishing spectrum condition with respect to $c$
but not the normally bounded spectrum condition?
\end{problem}

Related to Corollary \ref{annulus frag}, we pose the following problem.

\begin{problem}
Let $(\Sigma_g,\omega)$ be a closed Riemann surface of positive genus $g$ with a symplectic form $\omega$.
Let $C$ be a non-contractible simple closed curve in $\Sigma_g$
and $U$ an open subset of $\Sigma_g$ displaceable from $C$.
Does there exist a bi-Lipschitz injective homomorphism
\[
	I\colon(\mathbb{Z},\lvert\cdot\rvert)\to(\Ham(\Sigma_g),\|\cdot\|_U)\text{?}
\]
\end{problem}

The following problem is also related to Corollary \ref{annulus frag}.

\begin{problem}
Let $(\mathbb{T}^2=\mathbb{R}/\mathbb{Z}\times\mathbb{R}/\mathbb{Z},\omega)$ be a 2-torus with a symplectic form $\omega$.
Let $U$ be an open neighborhood of $(\{0\}\times\mathbb{R}/\mathbb{Z})\cup(\mathbb{R}/\mathbb{Z}\times\{0\})$
and $V$ a contractible open subset of $\mathbb{T}^2$ with $\mathbb{T}^2=U\cup V$.
We consider the open covering $\mathcal{U}=\{U,V\}$ of $\mathbb{T}^2$.
Does there exist a bi-Lipschitz injective homomorphism
\[
	I\colon(\mathbb{Z},\lvert\cdot\rvert)\to(\Ham(\mathbb{T}^2),\|\cdot\|_{\mathcal{U}})\text{?}
\]
\end{problem}

Related to Corollary \ref{cp2}, we pose the following problem.

\begin{problem}
Let $(\mathbb{C}P^n,\omega_{\mathrm{FS}})$ be $n$-dimensional complex projective space with the Fubini--Study form $\omega_{\mathrm{FS}}$.
Let $L_C$ be the Clifford torus in $\mathbb{C}P^n$
and $U$ an open subset of $\mathbb{C}P^n$ displaceable from $L_C$.
Let $(\Sigma_g,\omega)$ be a closed Riemann surface of positive genus $g$ with a symplectic form $\omega$
and $C$ a non-contractible simple closed curve in $\Sigma_g$.
We consider the product manifold $(\mathbb{C}P^n\times\Sigma_g,\omega_{\mathrm{FS}}\oplus\omega)$,
the Lagrangian submanifold $L_C\times C$,
and the open subset $\widehat{U}=U\times\Sigma_g$ of $\mathbb{C}P^n\times\Sigma_g$.
Does there exist a bi-Lipschitz injective homomorphism
\[
	I\colon(\mathbb{Z},\lvert\cdot\rvert)\to(\Ham(\mathbb{C}P^n\times\Sigma_g),\|\cdot\|_{\widehat{U}})\text{?}
\]
\end{problem}

By Proposition \ref{relative disp}, $\widehat{U}$ satisfies the bounded spectrum condition with respect to $c^{(L_C\times C;R)}$ for any ring $R$.
However, by an argument similar to \cite{EP09},
we see that $c^{(L_C\times C;R)}$ is not a quasi-morphism and that we therefore cannot apply Theorem \ref{theorem:main2}.


\section*{Acknowledgments}
The authors would like to thank Professor Yong-Geun Oh and Takahiro Matsushita for some advice.
Especially, Takahiro Matsushita read our draft seriously and gave a lot of comments on writing.
A part of this work was carried out while the first named author was visiting NCTS (Taipei, Taiwan)
and the second named author was visiting IBS-CGP (Pohang, Korea).
They would like to thank the institutes for their warm hospitality and support.


\bibliographystyle{amsart}

\begin{thebibliography}{10}
\bibitem[Ba]{Ba} A. Banyaga,
	\textit{Sur la structure du groupe des diff\'eomorphismes qui pr\'eservent une forme symplectique},
	Comment.\ Math.\ Helv.\ \textbf{53} (1978), no.~2, 174--227.
\bibitem[BC]{BC} P. Biran and O. Cornea,
	\textit{Rigidity and uniruling for Lagrangian submanifolds},
	Geom.\ Topol.\ \textbf{13} (2009), no.~5, 2881--2989.
\bibitem[BEP]{BEP} P. Biran, M. Entov and L. Polterovich,
	\textit{Calabi quasimorphisms for the symplectic ball},
	Commun.\ Contemp.\ Math.\ \textbf{6} (2004), no.~5, 793--802.
\bibitem[Br]{Br} M. Brandenbursky,
	\textit{Bi-invariant metrics and quasi-morphisms on groups of Hamiltonian diffeomorphisms of surfaces},
	Internat.\ J. Math.\ \textbf{26} (2015), no.~9, 1550066.
\bibitem[BK]{BK} M. Brandenbursky and J. K\c{e}dra,
	\textit{Fragmentation norm and relative quasimorphisms},
	\texttt{arXiv:1804.06592}.
\bibitem[BKS]{BKS} M. Brandenbursky, J. K\c{e}dra and E. Shelukhin,
	\textit{On the autonomous norm on the group of Hamiltonian diffeomorphisms of the torus},
	Commun.\ Contemp.\ Math.\ \textbf{20} (2018), no.~2, 1750042.
\bibitem[BIP]{BIP} D. Burago, S. Ivanov and L. Polterovich,
	\textit{Conjugation-invariant norms on groups of geometric origin},
	Adv.\ Stud.\ Pure Math.\ \textbf{52} (2008), 221--250.
\bibitem[Ca]{Ca} D. Calegari,
	\textit{scl},
	MSJ Memoirs, Vol.~20 (Mathematical Society of Japan, 2009).
\bibitem[CS]{CS} Y. Chekanov and F. Schlenk,
	\textit{Notes on monotone Lagrangian twist tori},
	Electron.\ Res.\ Announc.\ Math.\ Sci.\ \textbf{17} (2010), 104--121.
\bibitem[En]{E} M. Entov,
	\textit{Quasi-morphisms and quasi-states in symplectic topology},
	The Proceedings of the International Congress of Mathematicians (Seoul, 2014).
\bibitem[EP03]{EP03} M. Entov and L. Polterovich,
	\textit{Calabi quasimorphism and quantum homology},
	Int.\ Math.\ Res.\ Not.\ \textbf{2003} (2003), no.~30, 1635--1676.
\bibitem[EP06]{EP06} M. Entov and L. Polterovich,
	\textit{Quasi-states and symplectic intersections},
	Comment.\ Math.\ Helv.\ \textbf{81} (2006), no.~1, 75--99.
\bibitem[EP09]{EP09} M. Entov and L. Polterovich,
	\textit{Rigid subsets of symplectic manifolds},
	Compos.\ Math.\ \textbf{145} (2009), no.~3, 773--826.
\bibitem[EPP]{EPP} M. Entov, L. Polterovich and P. Py,
	\textit{On continuity of quasimorphisms for symplectic maps},
	in Perspectives in Analysis, Geometry, and Topology, eds.\ I. Itenberg, B. J\"oricke and M. Passare (Birkh\"auser/Springer, 2012), 169--197.
\bibitem[FOOO]{FOOO} K. Fukaya, Y.-G. Oh, H. Ohta and K. Ono,
	\textit{Spectral invariants with bulk, quasimorphisms and Lagrangian Floer theory},
	\texttt{arXiv:1105.5123v3}, to appear in Mem.\ Amer.\ Math.\ Soc.
\bibitem[Ga]{Gad} A. Gadbled,
	\textit{On exotic monotone Lagrangian tori in $\mathbb{CP}^2$ and $\mathbb{S}^2\times\mathbb{S}^2$},
	J. Symplectic Geom.\ \textbf{11} (2013), no.~3, 343--361.
\bibitem[Gr]{G} M. Gromov,
	\textit{Pseudo holomorphic curves in symplectic manifolds},
	Invent.\ Math.\ \textbf{82} (1985), no.~2, 307--347.
\bibitem[Ish]{Is} S. Ishikawa,
	\textit{Spectral invariants of distance functions},
	J. Topol.\ Anal.\ \textbf{8} (2016), no.~1, 655--676.
\bibitem[Ka16]{Ka3} M. Kawasaki,
	\textit{Fragmented Hofer's geometry},
	\texttt{arXiv:1612.01080}, to appear in Internat.\ J. Math.
\bibitem[Ka17]{Ka} M. Kawasaki,
	\textit{Superheavy Lagrangian immersions in surfaces},
	to appear in J. Symplectic Geom.
\bibitem[Ka18]{Ka2} M. Kawasaki,
	\textit{Function theoretical applications of Lagrangian spectral invariants},
	\texttt{arXiv:1811.00527}.
\bibitem[La]{L} S. Lanzat,
	\textit{Quasi-morphisms and symplectic quasi-states for convex symplectic manifolds},
	Int.\ Math.\ Res.\ Not.\ \textbf{2013} (2013), no.~23, 5321--5365.
\bibitem[LZ]{LZ} R. Leclercq and F. Zapolsky,
	\textit{Spectral invariants for monotone Lagrangians},
	J. Topol.\ Anal.\ \textbf{10} (2018), no.~3, 627--700.
\bibitem[LR]{LR} F. Le Roux,
	\textit{Simplicity of $\mathrm{Homeo}(\mathbb{D}^2,\partial\mathbb{D}^2,\mathrm{Area})$ and fragmentation of symplectic diffeomorphisms},
	J. Symplectic Geom.\ \textbf{8} (2010), no.~1, 73--93.
\bibitem[Mc]{Mc} D. McDuff,
	\textit{Monodromy in Hamiltonian Floer theory},
	Comment.\ Math.\ Helv.\ \textbf{85} (2010), no.~1, 95--133.
\bibitem[MVZ]{MVZ} A. Monzner, N. Vichery and F. Zapolsky,
	\textit{Partial quasimorphisms and quasistates on cotangent bundles, and symplectic homogenization},
	J. Mod.\ Dyn.\ \textbf{6} (2012), no.~2, 205--249.
\bibitem[OU]{OU} J. Oakley and M. Usher,
	\textit{On certain Lagrangian submanifolds of $S^2\times S^2$ and $\mathbb{C}P^n$},
	Algebr.\ Geom.\ Topol.\ \textbf{16} (2016), no.~1, 149--209.
\bibitem[Oh96]{Oh96} Y.-G. Oh,
	\textit{Relative Floer and quantum cohomology and the symplectic topology of Lagrangian submanifolds},
	in Contact and Symplectic Geometry, eds.\ C. B. Thomas (Cambridge, 1994), 201--267.
\bibitem[Oh05]{Oh05} Y.-G. Oh,
	\textit{Construction of spectral invariants of Hamiltonian paths on closed symplectic manifolds},
	in The Breadth of Symplectic and Poisson Geometry, eds.\ J. E. Marsden and T. Ratiu (Birkh\"auser/Springer, 2005), 525--570.
\bibitem[Oh06]{Oh06} Y.-G. Oh,
	\textit{Lectures on Floer theory and spectral invariants of Hamiltonian flows},
	in Morse Theoretic Methods in Nonlinear Analysis and in Symplectic Topology, eds.\ P. Biran, O. Cornea and F. Lalonde (Springer, 2006), 321--416.
\bibitem[Pon]{P} L. S. Pontryagin,
	\textit{Selected works Vol.~2, Topological groups},
	Edited and with a preface by R. V. Gamkrelidze.
	Translated from the Russian and with a preface by A. Brown.
	With additional material translated by P. S. V. Naidu.
	Third edition.
	Classics of Soviet Mathematics.
	Gordon \& Breach Science Publishers, New York, 1986.
\bibitem[Pol01]{Po01} L. Polterovich,
	\textit{The Geometry of the Group of Symplectic Diffeomorphism},
	Lectures in Mathematics, ETH Z\"urich (Birkh\"auser Basel, 2001).
\bibitem[Pol12]{Po} L. Polterovich,
	\textit{Quantum unsharpness and symplectic rigidity},
	Lett.\ Math.\ Phys.\ \textbf{102} (2012), no.~3, 245--264.
\bibitem[PR]{PR} L. Polterovich and D. Rosen,
	\textit{Function theory on symplectic manifolds},
	CRM Monograph Series, Vol.~34 (American Mathematical Society, 2014).
\bibitem[Py06a]{Py6} P. Py,
	\textit{Quasi-morphismes et invariant de Calabi},
	Ann.\ Sci.\ \'Ecole Norm.\ Sup.\ \textbf{4} (2006), no.~1, 177--195.
\bibitem[Py06b]{Py6-2} P. Py,
	\textit{Quasi-morphismes de Calabi et graphe de Reeb sur le tore},
	C. R. Math.\ Acad.\ Sci.\ Paris \textbf{343} (2006), no.~5, 323--328.
\bibitem[Sch]{Sch} M. Schwarz,
	\textit{On the action spectrum for closed symplectically aspherical manifolds},
	Pacific J. Math.\ \textbf{193} (2000), no.~2, 419--461.
\bibitem[Ush]{U} M. Usher,
	\textit{The sharp energy-capacity inequality},
	Commun.\ Contemp.\ Math.\ \textbf{12} (2010), no.~3, 457--473.
\bibitem[Wu]{W} W. Wu,
	\textit{On an exotic Lagrangian torus in $\mathbb{C}P^2$},
	Compos.\ Math.\ \textbf{151} (2015), no.~7, 1372--1394.
\bibitem[Zha]{Z} J. Zhang,
	\textit{Symplectic structure perturbations and continuity of symplectic invariants},
	\texttt{arXiv:1610.00516}.
\bibitem[Za]{Za} F. Zapolsky,
	\textit{The Lagrangian Floer-quantum-PSS package and canonical orientations in Floer theory},
	\texttt{arXiv:1507.02253}.
\end{thebibliography}

\end{document}